\documentclass[11pt]{article}

\usepackage{amsmath}
\usepackage{amsfonts}
\usepackage{amssymb}
\usepackage{amsthm}
\usepackage{bbm}
\usepackage{caption}
\usepackage{CJKutf8}
\usepackage{enumitem}
\usepackage{geometry}
\usepackage{graphicx}
\usepackage{import}
\usepackage{listings}
\usepackage{multicol}
\usepackage{nameref}
\usepackage{soul}
\usepackage{standalone}
\usepackage{titlesec}
\usepackage{tikz}
\usepackage{todonotes}
\usepackage{url}
\usepackage{xcolor}
\usepackage[many]{tcolorbox}
\usepackage[utf8]{inputenc}
\usepackage{authblk}
\usepackage[numbers,sort&compress]{natbib}
\usepackage{lineno}

\usepackage{hyperref} %

\newtheoremstyle{mytheoremstyle} %
    {\topsep}                    %
    {\topsep}                    %
    {\itshape}                   %
    {}                           %
    {\bfseries}                  %
    {:}                          %
    {.5em}                       %
    {}  %

\theoremstyle{mytheoremstyle}

\newtheorem{Theorem}{Theorem}
\newtheorem*{Theorem*}{Theorem}

\newtheorem{Lemma}{Lemma}
\newtheorem{Observation}{Observation}

\theoremstyle{definition}
\newtheorem{Definition}{Definition}
\newtheorem*{Definition*}{Informal Definition}

\newtcolorbox{problembox}[1]{
    tikznode boxed title,
    enhanced,
    arc=0mm,
    interior style={white},
    attach boxed title to top center= {yshift=-\tcboxedtitleheight/2},
    fonttitle=\bfseries,
    colbacktitle=white,coltitle=black,
    boxed title style={size=normal,colframe=white,boxrule=0pt},
    title={#1}
    }
    
\newcommand{\T}[0]{\mathcal{T}}

\newcommand{\Fr}{F_r}
\newcommand{\kr}{k_r}

\newcommand{\uMAF}{{\rm uMAF}}
\newcommand{\Fu}{{F}}
\newcommand{\ku}{k}

\newcommand{\FIG}[1]{{Figure \ref{FIG #1}}}

\newcommand{\LEM}[1]{{Lemma \ref{LEM #1}}}
\newcommand{\THM}[1]{{Theorem \ref{THM #1}}}

\usetikzlibrary{decorations.text}
\def\myshift#1{\raisebox{1ex}}

\newcommand{\vertex}{\node[vertex]}
\tikzstyle{vertex}=[draw, shape=circle, minimum size=0.1em, inner sep=1, fill]
\tikzstyle{every path}=[thick]

\colorlet{cyan}[rgb]{cyan}
\newcommand{\MyBlue}[0]{blue!50!cyan}
\newcommand{\MyGreen}[0]{green!75!blue}
\newcommand{\MyRed}[0]{red!85!cyan}

\newcommand{\MyOrange}[0]{orange!90!black}

\newcommand{\stevenrevise}[1]{\textcolor{black}{#1}}
\newcommand{\leorevise}[1]{\textcolor{black}{#1}}

\newcommand{\rvalue}[0]{r=\min\{\max\{\ku, 3\}, t+1\}}
\newcommand{\rrvalue}[0]{r=\min\{\max\{\kr, 3\}, t+1\}}
\newcommand{\uks}[0]{4tr\ku - 4tr - 3r\ku}

\newcommand{\rks}[0]{(2tr+1)(\kr-1)-r}

\title{A Kernel for the Maximum Agreement Forest Problem on Multiple Binary Phylogenetic Trees}
\author{S. Kelk, R. Meuwese, L.J.J. van Iersel \footnote{\url{steven.kelk@maastrichtuniversity.nl}, \url{L.J.J.vanIersel@tudelft.nl}, \url{R.H.Meuwese@tudelft.nl}. Kelk is at the Department of Advanced Computing Sciences, Maastricht University, The Netherlands and Meuwese and Van Iersel are at the Delft Institute of Applied Mathematics, Delft University of Technology, The Netherlands.}}
\date{}

\newcommand{\steven}[1]{\textcolor{black}{#1}}
\newcommand{\leo}[1]{#1}

\newcommand{\leonew}[1]{#1}
\newcommand{\leonewer}[1]{\textcolor{black}{#1}}

\newcommand{\stevenlater}[1]{\textcolor{black}{#1}}

\newcommand{\stevenfinal}[1]{\textcolor{black}{#1}}

\begin{document}

\maketitle

\begin{abstract}
The maximum agreement forest (MAF) problem in phylogenetics takes as input a set \stevenrevise{of} $t \geq 2$ binary phylogenetic trees $\T$ on the same set of taxa~$X$. It asks for a partition of~$X$ into the smallest number of blocks such that the subtrees induced by these blocks are disjoint and have common topology across all the trees in~$\T$. \stevenfinal{We produce a modified version of the well-known chain reduction rule in order to prove that after exhaustive application of reduction rules each tree has 
$O( t \cdot r \cdot k )$ leaves, where $k$ is the natural parameter (the number of blocks)
and $\rvalue$.} We prove this bound for both the unrooted and rooted version of the problem, and demonstrate that the bound $r$, the length to which common chains are truncated, is tight. Our results constitute the first kernels for MAF in the~$t>2$ regime. 
\end{abstract}

\section{Introduction}\label{SEC intro}

A phylogenetic tree $T$ on a set of species (or more generally, \emph{taxa}) $X$ is a tree whose leaves are bijectively labelled by $X$. Such trees are commonplace in the study of the evolution of species; interior nodes represent `branching' events in history, such as speciation events. Taken as a whole,~$T$ therefore represents
a hypothesis about how the species $X$ evolved from common ancestors \cite{SempleS03}. In this article we are only concerned with \emph{binary} phylogenetic trees, i.e. in which each interior node represents a bifurcation of lineages. These trees come in two variants: rooted, in which the direction of evolution is designated, and unrooted in which the direction of evolution has not yet been determined. 

Given a set of $t \geq 2$ phylogenetic trees $\T$, all on the same set $X$, the \emph{maximum agreement forest} problem \stevenrevise{(MAF)} asks for a partition of $X$ into the smallest number of blocks such that the subtrees induced by these blocks are disjoint and have common topology across all the trees in $\T$. MAF can be used as a measure of dissimilarity; %
this can be particularly useful if the tree inference process has constructed a set of competing tree hypotheses, not just one, and we wish to summarize their shared structure.

Unfortunately, the variants of MAF on rooted trees (rMAF) and unrooted trees (\uMAF) are both already NP-hard when $t=2$ \cite{AllenS01,HeinJWZ96,BordewichS04}. This has led to more than two decades of research aiming to tame this worst-case intractability, particularly in the area of parameterized complexity and approximation algorithms. Rather than list the full history of results in this area we refer to papers with detailed surveys of the literature such as \cite{BulteauW19,shi2018parameterized,shi2014algorithms,kelk2024deep}.

In this article our focus is on parameterized complexity. For $t=2$, both kernelization and branching algorithms are known. For \uMAF, a branching algorithm with running time $O^{*}(2.846^k)$ was recently proposed \cite{mestel2024split}, where here $k$ refers to the number of blocks in an optimal solution and `*' suppressess polynomial terms. The same article presents a branching algorithm for rMAF with running time $O^{*}(2.3391^k)$, very slightly improving upon the previous best of $O^{*}(2.3431^k)$~\cite{chen2015faster}. Linear kernels for {\uMAF} and rMAF were established in 2001 and 2004 respectively \cite{AllenS01,BordewichS04}, using two reduction rules known as the \emph{common subtree} and \emph{common chain} reduction rules. The state of the art kernel for rMAF augments this with a third reduction rule and (up to constant terms) achieves a kernel size of $9k$ \cite{kelk2023cyclic}. For \uMAF, augmenting the two original reduction rules with eight new rules also yields a kernel of size $9k$ \cite{kelk2024deep}. 

For $t>2$ progress has been comparatively slow. For rMAF, a branching algorithm exists that runs in time $O^{*}(2.42^k)$
\cite{shi2018parameterized}, and for {\uMAF} the best known branching algorithm runs in time
$O^{*}(4^k)$\cite{shi2014algorithms}. However, to date no \stevenrevise{explicit} kernelization results have been proven\footnote{\stevenrevise{The aforementioned branching algorithms in  
\cite{shi2018parameterized,shi2014algorithms} automatically imply the existence of a kernel whose size is \emph{exponential} in $k$. This leverages the well-known equivalence between the existence of fixed parameter tractable (FPT) algorithms and kernels, see e.g. Lemma 2.2 of \cite{DBLP:books/sp/CyganFKLMPPS15}.}
}.

In this article, we fill this gap by showing that the common subtree reduction rule, when combined with a modified version of the common chain reduction rule, \stevenfinal{reduces the number of leaves in each tree to} $O( t \cdot r \cdot k )$ where {$\rvalue$}
and $k$ is the natural parameter (the number of blocks). This bound holds for both {\uMAF} and rMAF. In proving this we have to forego the `generator' machinery that has been used in the recent $t=2$ kernelization literature to produce small kernels: this machinery breaks down for $t>2$. For this reason the analysis of the size of our kernel bears more resemblance to earlier work in this area, specifically 
\cite{AllenS01,BordewichS04}. The dependency on $t$ in the size of the kernel stems from the fact that the modified common chain rule truncates chains to length $r$, which is a function of $t$ and $k$; in contrast, in the $t=2$ kernelization literature truncation to constant length is sufficient. We show that our choice of $r$ is tight, i.e. there exist inputs in which shortening common chains to length less than $r$ alters the size of an optimum solution. \leonew{Since, for $t=2$, our bounds perfectly match the known bounds  under subtree and chain reduction, another contribution of this article is to provide a simpler proof of these kernelization results.}

Our kernel might be of interest to participants of the Parameterized Algorithms and Computational Experiments competition 2026 (PACE 2026), since the exact track of this competition asks participants to solve rMAF instances on potentially multiple trees \cite{pace2026}.

We conclude with a discussion of open problems and future research directions. \stevenrevise{Most notably, do rMAF and uMAF admit kernels whose size is at most \emph{polynomial} in \emph{only} $k$, as opposed to polynomial in $k$ and $t$, or exponential in $k$? }

\section{Preliminaries}\label{SEC pre}

        An \emph{unrooted binary phylogenetic tree} on $X$ is an undirected tree $T =(V(T),E(T))$ where every internal node has degree 3 and whose leaves \leo{(degree-1 nodes)} are bijectively labelled by a set of leaf labels (or more formally \emph{taxa}) $X$, where $|X|=n$.
        When it is clear from the context we will often refer to an unrooted binary phylogenetic tree as simply a tree.
        For every taxon $x\in X$ its \emph{parent} in a tree $T$ is the node %
        \leo{adjacent} to $x$. A \emph{cherry} in $T$ is a pair of distinct taxa $x, y \in X$ which have a common parent in $T$.
        
        We refer to $B \subseteq X$, with $|B| \geq 1$, as a \leo{\emph{block}}.
        For %
        \leo{a} block $B$ we write $T[B]$ to denote the %
        \leo{minimal subtree}
        of $T$ spanning exactly all the taxa in $B$. Furthermore we write $T|B$ to denote the %
        \leo{phylogenetic tree} obtained from $T[B]$ by suppressing nodes of degree 2. Subtree $T[B]$ is called \emph{pendant} in $T$ if there exists an edge such that deleting it splits $T$ into exactly two subtrees, where one tree has taxa $B$ and the other has \leo{taxa $X \setminus B$. The \emph{degree} of a block $B$ in a tree $T$, denoted as $deg^T(B)$, is the number of edges in~$T$ that have exactly one endpoint in~$T[B]$ (informally, the number of edges which need to be cut to split~$B$ from $X\setminus B$). Hence, $deg^T(B)=1$ precisely if~$T[B]$ is pendant.}
        
        For %
        \leorevise{$m\geq 1$}, let $C = (x_1, x_2 \ldots,x_m)$ be a sequence of distinct taxa in $X$. We call $C$ an \emph{$m$-chain} in $T$, or simply a \emph{chain} in $T$, if for each $i \in \{1, \ldots m-1\}$ the parent of $x_i$ is either equal to, or adjacent to, the parent of $x_{i+1}$\footnote{Due to the fact we are working with binary trees $x_i$ and $x_{i+1}$ can only share a parent if $i=1$ or $i=m-1$, and $\{x_i, x_{i+1}\}$ form a cherry.}.  Note that if
        $(x_1, x_2 \ldots,x_m)$ is an $m$-chain in $T$ then so is $(x_{m}, x_{m-1} \ldots,x_1)$.
        \leonewer{Chain $(x_1, \ldots,x_m)$ is \emph{maximal} if there is no chain $(y_1, \ldots,y_{m'})$ with $\{x_1, \ldots,x_m\}\subsetneq \{y_1, \ldots,y_{m'}\}$.}
        Examples of
        \leo{subtrees and chains}
        are given in \FIG{unrooted substructures}.

\begin{figure}[h]
    \begin{center}
    \begin{tikzpicture}[scale = 0.5]

    \node (F1) at (-3,0) {$T$};
    \vertex [label=above:$a$] (a) at (0,2) {};
    \vertex [label=left:$b$] (b) at (-1,1) {};
    \vertex [label=left:$c$] (c) at (-1,-1) {};
    \vertex [label=below:$d$] (d) at (0,-2) {};
    \vertex [label=below:$e$] (e) at (2,-1) {};
    \vertex [label=below:$f$] (f) at (3,-1) {};
    \vertex [label=below:$g$] (g) at (4,-1) {};
    \vertex [label=above:$h$] (h) at (6,2) {};
    \vertex [label=right:$i$] (i) at (7,1) {};
    \vertex [label=below:$j$] (j) at (5,-2) {};
    \vertex [label=below:$k$] (k) at (6,-3) {};
    \vertex [label=below:$l$] (l) at (7,-4) {};
    \vertex [label=below:$m$] (m) at (9,-4) {};
    
    \vertex (l1) at (0,-1){};
    \vertex (l2) at (1,0){};
    \vertex (l3) at (0,1){};
    \vertex (c1) at (2,0){};
    \vertex (c2) at (3,0){};
    \vertex (c3) at (4,0){};
    \vertex (r1) at (5,0){};
    \vertex (t1) at (6,1){};
    \vertex (r2) at (6,-1){};
    \vertex (r3) at (7,-2){};
    \vertex (r4) at (8,-3){};

    \draw
    (l1) edge (c)
    (l1) edge (d)
    (l2) edge (l1)
    (l2) edge (l3)
    (l3) edge (a)
    (l3) edge (b)
    (l2) edge (r1)
    (c1) edge (e)
    (c2) edge (f)
    (c3) edge (g)
    (r1) edge (t1)
    (t1) edge (h)
    (t1) edge (i)
    (r1) edge (r4)
    (r2) edge (j)
    (r3) edge (k)
    (r4) edge (l)
    (r4) edge (m)
    ;
    
    \end{tikzpicture}
    \caption{\leo{An unrooted binary phylogenetic tree~$T$ on~$X=\{a,\ldots,m\}$. For example,
    $T[\{a, b, c, d\}]$ is a pendant subtree, $(e, f, g)$ is a $3$-chain with $deg^T(\{e,f,g\})=2$ and $(j, k, l, m)$ is a $4$-chain in~$T$ with $deg^T(\{j,k,l,m\})=1$.}}
    \label{FIG unrooted substructures}
    \end{center}
\end{figure}

\leo{For leaf-labelled (not necessarily binary) trees~$T,T'$ both having leaf set~$X$,} we
        write $T = T'$ %
        if there is an isomorphism between $T$ and $T'$ that preserves the leaf labels~$X$. \leo{We are now ready to define the main structure studied in this paper.}

\begin{Definition}
        Let $\T$ be a set of unrooted phylogenetic trees on $X$. An \emph{(unrooted) agreement forest} $\Fu = \{B_1,B_2,\ldots,B_{\ku}\}$ for $\T$ is a partition of $X$ such that the following conditions hold. 
        \begin{enumerate}
            \item For all $i\in\{1,2,\ldots,{\ku}\}$, $T|B_i = T'|B_i$ for every pair of trees $T, T' \in \mathcal{T}$.
            \item For each pair $i,j\in\{1,2,\ldots,{\ku}\}$ with $i \neq j$, the subtrees $T[B_i]$ and $T[B_j]$ are vertex-disjoint in $T$ for every tree $T\in\T$.
        \end{enumerate}
\end{Definition}

\begin{figure}[h]
    \begin{center}
    \begin{tikzpicture}[scale = 0.5]

    \node (F1) at (-11,-5) {$T$};    
    \vertex [\MyGreen, label=left:$a$] (a) at (-9,-4) {};
    \vertex [\MyGreen, label=left:$b$] (b) at (-9,-6) {};
    \vertex [\MyGreen, label=below:$c$] (c) at (-8,-7) {};
    \vertex [\MyGreen, label=below:$d$] (d) at (-6,-7) {};
    \vertex [\MyBlue, label=right:$e$] (e) at (-5,-6) {};
    \vertex [\MyRed, label=right:$f$] (f) at (-5,-4) {};
    
    \vertex [\MyGreen] (l) at (-8,-5) {};
    \vertex [\MyGreen] (m) at (-7,-5) {};
    \vertex (r) at (-6,-5) {};
    \vertex [\MyGreen] (u) at (-7,-6) {};  

    \draw
    (m) edge [\MyGreen] (l)
    (m) edge (r)
    (m) edge [\MyGreen] (u)
    (l) edge [\MyGreen] (a)
    (l) edge [\MyGreen] (b)
    (u) edge [\MyGreen] (c)
    (u) edge [\MyGreen] (d)
    (r) edge (e)
    (r) edge (f)
    ;
    
    \node (F1) at (11,-5) {$T'$};
    \vertex [\MyGreen, label=left:$a$] (a) at (5,-4) {};
    \vertex [\MyGreen, label=right:$b$] (b) at (9,-4) {};
    \vertex [\MyGreen, label=below:$c$] (c) at (6,-7) {};
    \vertex [\MyGreen, label=below:$d$] (d) at (8,-7) {};
    \vertex [\MyBlue, label=left:$e$] (e) at (5,-6) {};
    \vertex [\MyRed, label=right:$f$] (f) at (9,-6) {};
    
    \vertex [\MyGreen] (l) at (6,-5) {};
    \vertex [\MyGreen] (m) at (7,-5) {};
    \vertex [\MyGreen] (r) at (8,-5) {};
    \vertex [\MyGreen] (u) at (7,-6) {}; 

    \draw
    (m) edge [\MyGreen] (l)
    (m) edge [\MyGreen] (r)
    (m) edge [\MyGreen] (u)
    (l) edge [\MyGreen] (a)
    (l) edge (e)
    (u) edge [\MyGreen] (c)
    (u) edge [\MyGreen] (d)
    (r) edge [\MyGreen] (b)
    (r) edge (f)
    ;
    
    \end{tikzpicture}
    \caption{Example of an unrooted agreement forest  $\Fu = \{\{a,b,c,d\}, \{e\}, \{f\} \}$ for phylogenetic trees $T$ and $T'$. This is a maximum agreeement forest: a smaller agreement forest is not possible.}
    \label{FIG example umaf}
    \end{center}
\end{figure}
        The \emph{size} of $\Fu$ is simply~$|\Fu|$, i.e., the number of blocks in the partition.
        An agreement forest with minimum \leo{size} is called a \emph{maximum} agreement forest (MAF), see Figure \ref{FIG example umaf} for an example. In this paper, we consider the following problem:

\begin{problembox}{The Unrooted Maximum Agreement Forest Problem}
\textbf{Input:} A set $\T$ of unrooted binary phylogenetic trees on $X$ and an integer~$\ku$.

\textbf{Question:} \textit{Does there esist an unrooted agreement forest $\Fu$ for $\T$ of size at most $\ku$?}%

\end{problembox}

There is also a \emph{rooted} version of the agreement forest problem. A \emph{rooted} binary phylogenetic tree on $X$ is a tree with a single indegree-0 and outdegree-2 \emph{root}, where the leaves are bijectively labelled by $X$ and all other nodes have indegree 1 and outdegree 2. Edges are directed away from the root. If $|X|=1$ then the single node labelled with the unique element of $X$ is regarded as a rooted binary phylogenetic tree.
In this rooted context $T = T'$ if there is an isomorphism between $T$ and $T'$ that respects the direction of the edges (and, as usual, preserves the leaf labels $X$). Core definitions, such as $T[B]$ and $T|B$, go through unchanged, \stevenrevise{with one crucial exception: when constructing $T|B$
from $T[B]$, the unique node of $T[B]$ with indegree-0 and outdegree-2 (i.e. its root) is not suppressed.} The substructures for rooted trees are similar to their unrooted counterparts, as shown in \FIG{rooted substructures}. However, chains are \stevenlater{slightly} different. In a rooted phylogenetic tree chains have an orientation, starting with the taxon closest to the root. For example in \FIG{rooted substructures} $(g,f,e)$ is a chain of $T$, but $(e,f,g)$ is not.

\begin{figure}[h]
    \begin{center}
    \begin{tikzpicture}[scale = 0.5]

    \node (F1) at (3,2) {$T$};
    \vertex [label=below:$a$] (a) at (-2,-6) {};
    \vertex [label=below:$b$] (b) at (0,-6) {};
    \vertex [label=below:$c$] (c) at (2,-6) {};
    \vertex [label=below:$d$] (d) at (4,-6) {};
    \vertex [label=below:$e$] (e) at (3,-3) {};
    \vertex [label=below:$f$] (f) at (4,-2) {};
    \vertex [label=below:$g$] (g) at (5,-1) {};
    \vertex [label=below:$h$] (h) at (7,-1) {};
    \vertex [label=below:$i$] (i) at (8,-2) {};
    \vertex [label=below:$j$] (j) at (9,-3) {};
    \vertex [label=below:$k$] (k) at (11,-3) {};

    \vertex (t1) at (6,2) {};
    \vertex (r1) at (8,0) {};
    \vertex (r2) at (9,-1) {};
    \vertex (r3) at (10,-2) {};
    
    \vertex (c1) at (4,0) {};
    \vertex (c2) at (3,-1) {};
    \vertex (c3) at (2,-2) {};
    \vertex (l1) at (1,-3) {};
    \vertex (l2) at (-1,-5) {};
    \vertex (l3) at (3,-5) {};

    \draw [thin]
    (t1) edge [->] (r1)
    (r1) edge [->] (r2)
    (r2) edge [->] (r3)
    (t1) edge [->] (c1)
    (c1) edge [->] (c2)
    (c2) edge [->] (c3)
    (c3) edge [->] (l1)
    (l1) edge [->] (l2)
    (l1) edge [->] (l3)
    
    (r1) edge [->] (h)
    (r2) edge [->] (i)
    (r3) edge [->] (j)
    (r3) edge [->] (k)
    (c1) edge [->] (g)
    (c2) edge [->] (f)
    (c3) edge [->] (e)
    (l2) edge [->] (a)
    (l2) edge [->] (b)
    (l3) edge [->] (c)
    (l3) edge [->] (d)
    ;
    
    \end{tikzpicture}
    \caption{\leo{A rooted binary phylogenetic tree~$T$. For example, $T[\{a, b, c, d\}]$ is a pendant subtree, $(g, f, e)$ is a 3-chain and $(h, i, j, k)$ a 4-chain in~$T$.}}
    \label{FIG rooted substructures}
    \end{center}
\end{figure}

Let $\T$ be a set of rooted binary phylogenetic trees on $X$. A \emph{(rooted) agreement forest} \leo{for~$\T$} is defined identically to the unrooted case; the only difference is the different semantics of isomorphism in the $T|B_i = T'|B_i$ condition. \leo{See Figure~\ref{FIG example rmaf} for an example.} 

\begin{figure}[h]
    \begin{center}
    \begin{tikzpicture}[scale = 0.5]

    \node (F1) at (-9,-3) {$T$};
    
    \vertex [\MyGreen, label=below:$a$] (a) at (-10,-6) {};
    \vertex [\MyGreen, label=below:$b$] (b) at (-6,-6) {};
    \vertex [\MyGreen, label=below:$c$] (c) at (-4,-6) {};
    \vertex [\MyBlue, label=below:$d$] (d) at (-9,-7) {};
    \vertex [\MyBlue, label=below:$e$] (e) at (-7,-7) {};
    
    \vertex [\MyGreen] (l) at (-8,-4) {};
    \vertex [\MyGreen] (r) at (-7,-5) {};
    \vertex [\MyGreen] (t) at (-7,-3) {};
    \vertex [\MyBlue] (de) at (-8,-6) {};  

    \draw
    (t) edge [\MyGreen, ->] (l)
    (t) edge [\MyGreen, ->] (c)
    (l) edge [\MyGreen, ->] (a)
    (l) edge [\MyGreen, ->] (r)
    (r) edge [->] (de)
    (r) edge [\MyGreen, ->] (b)
    (de) edge [\MyBlue, ->] (d)
    (de) edge [\MyBlue, ->] (e)
    ;
    
    \node (F1) at (9,-3) {$T'$};
    
    \vertex [\MyGreen, label=below:$a$] (a) at (4,-6) {};
    \vertex [\MyGreen, label=below:$b$] (b) at (6,-6) {};
    \vertex [\MyGreen, label=below:$c$] (c) at (7,-5) {};
    \vertex [\MyBlue, label=below:$d$] (d) at (8,-6) {};
    \vertex [\MyBlue, label=below:$e$] (e) at (10,-6) {};
    
    \vertex [\MyGreen] (l) at (5,-5) {};
    \vertex [\MyGreen] (r) at (8,-4) {};
    \vertex [\MyGreen] (t) at (7,-3) {};
    \vertex [\MyBlue] (de) at (9,-5) {};  

    \draw
    (t) edge [\MyGreen, ->] (l)
    (t) edge [\MyGreen, ->] (r)
    (l) edge [\MyGreen, ->] (a)
    (l) edge [\MyGreen, ->] (b)
    (r) edge [->] (de)
    (r) edge [\MyGreen, ->] (c)
    (de) edge [\MyBlue, ->] (d)
    (de) edge [\MyBlue, ->] (e)
    ;
    
    \end{tikzpicture}
    \caption{Example of a rooted agreement forest $\Fr = \{\{a,b,c\}, \{d, e\} \}$ for the rooted phylogenetic trees $T$ and $T'$. This is a maximum agreement forest because an agreement forest with fewer blocks is not possible.}
    \label{FIG example rmaf}
    \end{center}
\end{figure}

Now the problem becomes:

\begin{problembox}{The Rooted Maximum Agreement Forest Problem}
\textbf{Input:} A set $\T$ of rooted binary phylogenetic trees on $X$ and an integer~$\kr$.

\textbf{Question:} \textit{Does there exist a rooted agreement forest $\Fr$ for $\T$ of size at most~$\kr$?}%

\end{problembox}

Where necessary we will write uMAF and rMAF to distinguish the unrooted and rooted versions of the problem. 

We will use the following
observation freely in the remainder of the article, which holds for both uMAF and rMAF, and which follows directly from the definitions. It shows that the size of a maximum agreement forest is non-increasing under the action of deleting taxa from the trees.

\begin{Observation}
Let $\T$ be a set of (rooted or unrooted) binary phylogenetic trees on~$X$ and $F$ an agreement forest for $\T$. If $X' \subseteq X$ then $$F' = \{ B \cap X' \mid B \in F\leo{, B\cap X'\neq\emptyset} \}$$ is an agreement forest for $\mathcal{T}' = \{ T|X' \mid T \in \T \}$ and $|F'| \leq |F|$.
\label{obs:closed}
\end{Observation}

In the main part of the article we will focus solely on uMAF. However, as we will point out later in Section \ref{subsec:rootedkern} the results go through almost entirely unchanged for rMAF.

\section{Reduction Rules}

    Let $\T$ be a set of unrooted binary phylogenetic trees on $X$ and~$S\subseteq X$. We say that~$S$ \emph{induces a common pendant subtree} of~$\T$ if \steven{$T[S]$ is pendant in every tree $T \in \T$} and \leo{
    $T[S]=T'[S]$ for each pair of trees~$T,T'\in\T$}\footnote{Note the use of $T[.]$ here: the pendant subtree should be `rooted' at the same location in every tree in $\T$.}.
    \leo{If~$S$ induces a common pendant subtree of~$\T$, then the \emph{subtree reduction rule} picks an arbitrary~$s\in S$ and replaces each tree~$T\in\T$ by~$T|((X\setminus S)\cup\{s\})$, see Figure~\ref{FIG example subtree reduction}.}
        
\begin{figure}[h]
    \begin{center}
    \begin{tikzpicture}[scale = 0.5]

    \node (F1) at (-4,0) {$T$};    
    \vertex [label=left:$a$] (a) at (-2,1) {};
    \vertex [label=left:$b$] (b) at (-2,-1) {};
    \vertex [label=below:$c$] (c) at (-1,-2) {};
    \vertex [label=below:$d$] (d) at (1,-2) {};
    \vertex [label=right:$e$] (e) at (2,-1) {};
    \vertex [label=right:$f$] (f) at (2,1) {};
    
    \vertex [] (l) at (-1,0) {};
    \vertex [] (m) at (0,0) {};
    \vertex [] (r) at (1,0) {};
    \vertex [] (u) at (0,-1) {};  

    \draw
    (m) edge [] (l)
    (m) edge [] (r)
    (m) edge [] (u)
    (l) edge [] (a)
    (l) edge [] (b)
    (u) edge [] (c)
    (u) edge [] (d)
    (r) edge [] (e)
    (r) edge [] (f)
    ;
    
    \node (F1) at (-4,-5) {$T'$};
    \vertex [label=left:$a$] (a) at (-2,-4) {};
    \vertex [label=right:$b$] (b) at (2,-6) {};
    \vertex [label=below:$c$] (c) at (-1,-7) {};
    \vertex [label=below:$d$] (d) at (1,-7) {};
    \vertex [label=left:$e$] (e) at (-2,-6) {};
    \vertex [label=right:$f$] (f) at (2,-4) {};
    
    \vertex [] (l) at (-1,-5) {};
    \vertex [] (m) at (0,-5) {};
    \vertex [] (r) at (1,-5) {};
    \vertex [] (u) at (0,-6) {};  

    \draw
    (m) edge [] (l)
    (m) edge [] (r)
    (m) edge [] (u)
    (l) edge [] (a)
    (l) edge [] (e)
    (u) edge [] (c)
    (u) edge [] (d)
    (r) edge [] (b)
    (r) edge [] (f)
    ;

    \node (F1) at (16,0) {$T_S$};    
    \vertex [label=left:$a$] (a) at (10,1) {};
    \vertex [label=left:$b$] (b) at (10,-1) {};
    \vertex [label=right:$e$] (e) at (14,-1) {};
    \vertex [label=right:$f$] (f) at (14,1) {};
    
    \vertex [] (l) at (11,0) {};
    \vertex [] (m) at (12,0) {};
    \vertex [] (r) at (13,0) {};
    \vertex [label=below:$c$] (u) at (12,-1) {};  

    \draw
    (m) edge [] (l)
    (m) edge [] (u)
    (m) edge [] (r)
    (l) edge [] (a)
    (l) edge [] (b)
    (r) edge [] (e)
    (r) edge [] (f)
    ;
    
    \node (F1) at (16,-5) {$T_S'$};
    \vertex [label=left:$a$] (a) at (10,-4) {};
    \vertex [label=right:$b$] (b) at (14,-6) {};
    \vertex [label=left:$e$] (e) at (10,-6) {};
    \vertex [label=right:$f$] (f) at (14,-4) {};
    
    \vertex [] (l) at (11,-5) {};
    \vertex [] (m) at (12,-5) {};
    \vertex [] (r) at (13,-5) {};
    \vertex [label=below:$c$] (u) at (12,-6) {};  

    \draw
    (m) edge [] (l)
    (m) edge [] (r)
    (m) edge [] (u)
    (l) edge [] (a)
    (l) edge [] (e)
    (r) edge [] (b)
    (r) edge [] (f)
    ;
    
    \draw 
    (3, 0) edge [out=0, in=180, looseness=1] (5, -1.25)
    (3, -5) edge [out=0, in=180, looseness=1] (5, -3.75)

    (5,-1.25) edge node [label=below:\small{Subtree}] {} (7,-1.25)
    (5,-3.75) edge node [label=above:\small{Reduction}] {} (7,-3.75)
    
    (7, -1.25) edge [->, out=0, in=180, looseness=1] (9, 0)
    (7, -3.75) edge [->, out=0, in=180, looseness=1] (9, -5)
    ;
        
    \end{tikzpicture}
    \caption{Example of applying the subtree reduction rule on $\T = \{T, T'\}$ with $S=\{c,d\}$, reducing the common pendant subtree to a single taxon $c$.}
    \label{FIG example subtree reduction}
    \end{center}
\end{figure}

The following lemma is folklore in the phylogenetics literature; correctness for two trees is straightforward (see \cite{AllenS01,BordewichS04} for related discussions) and extends without effort to three or more trees.

\begin{Lemma}\label{LEM safe subtree reduction}
    Let $\T$ be a set of unrooted binary phylogenetic trees on~$X$ and let $\T_S$ be the result of applying the subtree reduction rule. Then there exists an unrooted agreement forest of $\T_S$ of size at most $\ku$ if and only if $\T$ has an unrooted agreement forest of size at most $k$.
\end{Lemma}
    
    There is another well-known reduction rule called the \emph{chain reduction} rule. It is known that when the input consists of precisely two trees, %
    we can reduce a common chain of length 4 or more to a common 3-chain. When combined with the subtree reduction rule this is sufficient to obtain a kernel of size $15\ku$ \cite{KelkL18}. However, there is no chain reduction rule known for more than two trees. In the next section we extend the chain reduction to multiple trees and then use it to produce a kernel.

\subsection{Chain Reduction for Multiple Trees}
    Let $\T$ be a set of unrooted binary phylogenetic trees on $X$.
    We say that a %
    \leo{sequence of taxa}~$C$ is a \emph{common chain} of $\T$ if $C$ is a chain in each $T\in\T$. The intuition behind the following reduction rule is to
    shorten long common chains without altering the size of an optimal solution. 
 Let $C' = (x_1, \ldots, x_m)$  be a common $m$-chain where \leo{$m > r$} and $\rvalue$. 
 Here,~$k$ is the target parameter, i.e. the value $k$ in the question ``Is there an agreement forest with at most~$k$ blocks?''
 Note that $r \geq 3$, because $t \geq 2$. The \emph{chain reduction rule} reduces every tree $T\in\T$ by removing all but $r$ taxa \leo{from the chain}. More formally: let $X_C = (X\setminus \{x_{r+1}, x_{r+2} \ldots, x_m\})$, then \leo{replace} each $T \in \T$ \leo{by} $T_C=T|X_C$. \stevenrevise{(Note that deleting any $m-r$ taxa from $C'$ has the same effect; there is nothing special about deleting the last $m-r$ taxa).}

\begin{figure}[h]
    \begin{center}
    \begin{tikzpicture}[scale = 0.5]

    \node (F1) at (-3.5,0) {$T$};    
    \vertex [label=above:$a$] (a) at (-2,0) {};
    \vertex [label=below:$b$] (b) at (-1,-1) {};
    \vertex [label=below:$c$] (c) at (0,-1) {};
    \vertex [label=below:$d$] (d) at (1,-1) {};
    \vertex [label=below:$e$] (e) at (2,-1) {};
    \vertex [label=below:$f$] (f) at (3,-1) {};
    \vertex [label=above:$g$] (g) at (4,0) {};
    
    \vertex (l1) at (-1,0) {}; 
    \vertex (l2) at (0,0) {};
    \vertex (l3) at (1,0) {};
    \vertex (l4) at (2,0) {};
    \vertex (l5) at (3,0) {};
    
    \draw
    (l1) edge (l5)
    (l1) edge (a)
    (l1) edge (b)
    (l2) edge (c)
    (l3) edge (d)
    (l4) edge (e)
    (l5) edge (f)
    (l5) edge (g)
    ;
    
    \node (F1) at (-3.5,-5) {$T'$};
    \vertex [label=above:$a$] (a) at (4,-5) {};
    \vertex [label=below:$b$] (b) at (-1,-6) {};
    \vertex [label=below:$c$] (c) at (0,-6) {};
    \vertex [label=below:$d$] (d) at (1,-6) {};
    \vertex [label=below:$e$] (e) at (2,-6) {};
    \vertex [label=below:$f$] (f) at (3,-6) {};
    \vertex [label=above:$g$] (g) at (-2,-5) {};
    
    \vertex (l1) at (-1,-5) {}; 
    \vertex (l2) at (0,-5) {};
    \vertex (l3) at (1,-5) {};
    \vertex (l4) at (2,-5) {};
    \vertex (l5) at (3,-5) {};

    \draw
    (l1) edge (l5)
    (l1) edge (g)
    (l1) edge (b)
    (l2) edge (c)
    (l3) edge (d)
    (l4) edge (e)
    (l5) edge (f)
    (l5) edge (a)
    ;

    \node (F1) at (18.5,0) {$T_C$};    
    \vertex [label=above:$a$] (a) at (13,0) {};
    \vertex [label=below:$b$] (b) at (14,-1) {};
    \vertex [label=below:$c$] (c) at (15,-1) {};
    \vertex [label=below:$d$] (d) at (16,-1) {};
    \vertex [label=above:$g$] (g) at (17,0) {};
    
    \vertex (l1) at (14,0) {}; 
    \vertex (l2) at (15,0) {};
    \vertex (l5) at (16,0) {};

    \draw
    (l1) edge (l5)
    (l1) edge (a)
    (l1) edge (b)
    (l2) edge (c)
    (l5) edge (d)
    (l5) edge (g)
    ;
    
    \node (F1) at (18.5,-5) {$T_C'$};
    \vertex [label=above:$a$] (a) at (17,-5) {};
    \vertex [label=below:$b$] (b) at (14,-6) {};
    \vertex [label=below:$c$] (c) at (15,-6) {};
    \vertex [label=below:$d$] (d) at (16,-6) {};
    \vertex [label=above:$g$] (g) at (13,-5) {};
    
    \vertex (l1) at (14,-5) {}; 
    \vertex (l2) at (15,-5) {};
    \vertex (l5) at (16,-5) {};

    \draw
    (l1) edge (l5)
    (l1) edge (g)
    (l1) edge (b)
    (l2) edge (c)
    (l5) edge (d)
    (l5) edge (a)
    ;
    
    \draw 
    (5, 0) edge [out=0, in=180, looseness=0.8] (7.5, -1.25)
    (5, -5) edge [out=0, in=180, looseness=0.8] (7.5, -3.75)
    
    (7.5,-1.25) edge node [label=below:\small{Chain}] {} (9.5,-1.25)
    (7.5,-3.75) edge node [label=above:\small{Reduction}] {} (9.5,-3.75)
    
    (9.5, -1.25) edge [->, out=0, in=180, looseness=0.8] (11.5, 0)
    (9.5, -3.75) edge [->, out=0, in=180, looseness=0.8] (11.5, -5)
    ;
    
    \end{tikzpicture}
    \caption{Example of applying the chain reduction rule on $\T = \{T, T'\}$ where $C'=(b, c, d, e, f)$ is a long common chain. We reduce it to the smaller chain $C=(b, c, d)$.}
    \label{FIG example chain reduction}
    \end{center}
\end{figure}
When $t=2$ we get the well-known reduction rule which reduces a common chain to a common 3-chain. This is proven in \cite{AllenS01} to be safe. In the next lemma we prove safeness for $t \geq 2$.

\begin{Lemma}\label{LEM safe chain reduction}
    Let $\T$ be a set of~\leorevise{$t$} unrooted binary phylogenetic trees on $X$ %
    and~$C'$ a common $m$-chain of~$\T$ with %
    \leorevise{$m>\rvalue$}.
    Let $\T_C$ be the result of applying the chain reduction rule with respect to~$C'$. Then there exists an unrooted agreement forest of $\T_C$ of size at most $\ku$ if and only if $\T$ has an unrooted agreement forest of size at most $\ku$.
\end{Lemma}

\begin{proof}
Let $\Fu$ be an unrooted agreement forest for $\T$ of size $k$. It then follows from Observation \ref{obs:closed} (take $X'$ equal to $X_C$) that there exists an unrooted agreement forest of $\T_C$ of size at most $k$. This completes one direction of the proof.

To prove the other direction, suppose there exists an unrooted agreement forest $\Fu_C$ of $\T_C$ of size at most $\ku$. We say that a chain is \emph{preserved} by an unrooted agreement forest if all its elements are in the same block of the unrooted agreement forest. \leo{Let~$C$ be the truncated common chain~$(x_1,\ldots ,x_r)$.}
For convenience we overload $C$ to \leo{also} denote the set of taxa $\{x_1, \ldots, x_r\}$. %
Suppose that $C$ is not preserved in $\Fu_C$. 

We will prove that there exists an unrooted agreement forest $\Fu_C'$ of $\T_C$ of size at most $\ku$ in which $C$ is preserved and contained in block $B_C$. 
From this (and the fact that $r \geq 3$) it will follow that
the taxa $x_{r+1},\ldots ,x_m$ can simply be added to the block $B_C$ to obtain an agreement forest of size at most $k$ for the original set of trees $\T$\footnote{It is well-known that for both rooted and unrooted binary phylogenetic trees common chains of length at least 3 can be extended in this way; informally, three taxa are sufficient to \stevenlater{ensure that the truncated chain is `oriented' the same way in every tree in $\T_C$.}}.

We call a block~$B\in\Fu_C$ an \emph{inside-out} block if $B$ contains at least one taxon of $C$ and at least one taxon not in $C$. We claim that $\Fu_C$ has at most two inside-out blocks. To see this, consider any $T\in\T_C$ and observe that there are at most two edges that are incident to a vertex of $T[X_C]$ and to a vertex not in $T[X_C]$. Let $e_1,e_r$ be these edges (if they exist), where $e_1$ is incident to the parent of $x_1$ and $e_r$ is incident to the parent of $x_r$. For each inside-out block $B$, $T[B]$ contains at least one of $e_1,e_r$. By the definition of an unrooted agreement forest, for any two blocks $B_1, B_2\in\Fu_C$, the subtrees $T[B_1]$ and $T[B_2]$ are vertex-disjoint and hence edge-disjoint. It follows that there can be at most two inside-out blocks. We will split into three cases: two, one or zero inside-out blocks.

The first case we consider is that $\Fu_C$ \textbf{has exactly two} inside-out blocks $B$, $B'$. In this case, we merge the two inside-out blocks and %
\leo{any} blocks that are completely contained within $C$ %
into a single block \stevenrevise{$B_C$}, giving partition $\Fu_C'$ of $X_C$. \stevenrevise{This is depicted schematically in Figure \ref{fig:twoblocks}, where for simplicity we assume $C=(x_1, x_2, x_3, x_4, x_5)$.}

\begin{figure}[h]
    \begin{center}
    \begin{tikzpicture}[scale = 0.5]

    \node (F1) at (-4,0) {$T$};  
    \node (F2) at (0,1) {$B$};
    \node (F3) at (4,1) {$B'$};
    \vertex [label=below:$x_1$] (a) at (0,-1) {};
    \vertex [\MyGreen,label=below:$x_2$] (b) at (1,-1) {};
    \vertex [\MyGreen,label=below:$x_3$] (c) at (2,-1) {};
    \vertex [label=below:$x_4$] (d) at (3,-1) {};
    \vertex [\MyBlue, label=below:$x_5$] (e) at (4,-1) {};
    
    \vertex [\MyGreen] (l1) at (0,0) {}; 
    \vertex [\MyGreen] (l2) at (1,0) {};
    \vertex [\MyGreen] (l3) at (2,0) {};
    \vertex (l4) at (3,0) {};
    \vertex [\MyBlue] (l5) at (4,0) {};
    
    \draw
    (l1) edge[\MyGreen] (l3)
    (l3) edge (l4)
    (l4) edge (l5)
    (l1) edge (a)
    (l2) edge [\MyGreen](b)
    (l3) edge [\MyGreen] (c)
    (l4) edge (d)
    (l5) edge [\MyBlue](e)
    (-2,0) edge node [label=below:$e_1$] {} (l1)
    (-2,0) edge[\MyGreen] (l1)
    (l5) edge node [label=below:$e_5$] {} (6,0)
    (l5) edge [\MyBlue] (6,0)
    ;
    
    \filldraw[fill=black] (-2,0) -- (-3,1) -- (-3,-1) -- (-2,0);
    \filldraw[fill=black] (6,0) -- (7,1) -- (7,-1) -- (6,0);
    
    \begin{scope} [shift={(8,0)}]
        \draw (0,0) edge [->] (2,0);
    \end{scope}
    
    \begin{scope}[shift={(15,0)}]
    \node (F1) at (-4,0) {$T$};  
    \node (F2) at (2,1) {$B_C$};  
    \vertex [\MyRed,label=below:$x_1$] (a) at (0,-1) {};
    \vertex [\MyRed,label=below:$x_2$] (b) at (1,-1) {};
    \vertex [\MyRed,label=below:$x_3$] (c) at (2,-1) {};
    \vertex [\MyRed,label=below:$x_4$] (d) at (3,-1) {};
    \vertex [\MyRed, label=below:$x_5$] (e) at (4,-1) {};
    
    \vertex [\MyRed] (l1) at (0,0) {}; 
    \vertex [\MyRed] (l2) at (1,0) {};
    \vertex [\MyRed] (l3) at (2,0) {};
    \vertex [\MyRed] (l4) at (3,0) {};
    \vertex [\MyRed] (l5) at (4,0) {};
    
    \draw
    (l1) edge [\MyRed] (l5)
    (l1) edge [\MyRed] (a)
    (l2) edge [\MyRed](b)
    (l3) edge [\MyRed] (c)
    (l4) edge [\MyRed] (d)
    (l5) edge [\MyRed](e)
    (-2,0) edge node [label=below:$e_1$] {} (l1)
    (-2,0) edge[\MyRed] (l1)
    (l5) edge node [label=below:$e_5$] {} (6,0)
    (l5) edge [\MyRed] (6,0)
    ;
    
    \filldraw[fill=black] (-2,0) -- (-3,1) -- (-3,-1) -- (-2,0);
    \filldraw[fill=black] (6,0) -- (7,1) -- (7,-1) -- (6,0);
    \end{scope}
    \end{tikzpicture}
    \caption{\stevenrevise{If $F_C$ has two inside-out blocks $B$ (green) and $B'$ (blue) then in every tree $T \in \T_C$, $T[B]$ and $T[B']$ enter the chain in the same way. In this example $T[B]$ will always use $e_1$ but not $e_5$, and $T[B']$ will always use $e_5$ but not $e_1$, in every $T \in \T_C$. This allows $B$, $B'$ and the remainder of the chain $C$ to be merged together into a single block $B_C$ (red).}}
    \label{fig:twoblocks}
    \end{center}
\end{figure}
\vspace{0.4cm}
Clearly the size of $\Fu_C'$ is at most the size of $\Fu_C$. We now argue that $\Fu_C'$ is an unrooted agreement forest of $\T_C$. $\Fu_C'$ satisfies the non-overlapping condition because any block that would overlap with $B_C$ would also overlap with $B$ \stevenlater{or}
$B'$ contradicting that $\Fu_C$ is an unrooted agreement forest.
To see that the topological equality condition is also satisfied, consider the two edges $e_1,e_r$ defined above. Assume without loss of generality that $T[B]$ uses $e_1$ for some $T\in\T_C$. Then $T[B]$ uses $e_1$ and $T[B']$ uses $e_r$ in each tree $T\in\T_C$. From this together with the notion that $C$ has the exact same topology in every tree $T\in \stevenlater{\T_C}$ it follows that for the merged block $B_C$ the equality $T|B_C= T'|B_C$ holds for all $T,T'\in\T_C$.

    The second case is that $\Fu_C$ has \textbf{exactly one inside-out block} $B$. If $|B\cap C|\geq 3$ then observe that either $T[B]$ uses $e_1$ (but not $e_r$) in every tree $T \in \T_C$, $e_r$ (but not $e_1$) in every tree $T \in \T_C$, or both $e_1$ and $e_r$ in every tree $T \in \T_C$. Whichever situation applies, merging $B$ with all other blocks that intersect $C$ yields the desired agreement forest. \stevenrevise{This is depicted in Figure
    \ref{fig:one_insideout_size3ormore}.}

\begin{figure}[h]
    \begin{center}
    \begin{tikzpicture}[scale = 0.5]

    \node (F1) at (-4,0) {$T$};  
    \node (F2) at (0,1) {$B$};
    \vertex [\MyGreen, label=below:$x_1$] (a) at (0,-1) {};
    \vertex [\MyGreen, label=below:$x_2$] (b) at (1,-1) {};
    \vertex [\MyGreen, label=below:$x_3$] (c) at (2,-1) {};
    \vertex [label=below:$x_4$] (d) at (3,-1) {};
    \vertex [label=below:$x_5$] (e) at (4,-1) {};
    
    \vertex [\MyGreen] (l1) at (0,0) {}; 
    \vertex [\MyGreen] (l2) at (1,0) {};
    \vertex [\MyGreen] (l3) at (2,0) {};
    \vertex (l4) at (3,0) {};
    \vertex (l5) at (4,0) {};
    
    \draw
    (l1) edge[\MyGreen] (l3)
    (l3) edge (l4)
    (l4) edge (l5)
    (l1) edge [\MyGreen] (a)
    (l2) edge [\MyGreen] (b)
    (l3) edge [\MyGreen] (c)
    (l4) edge (d)
    (l5) edge (e)
    (-2,0) edge node [label=below:$e_1$] {} (l1)
    (-2,0) edge[\MyGreen] (l1)
    (l5) edge node [label=below:$e_5$] {} (6,0)
    (l5) edge (6,0)
    ;
    
    \filldraw[fill=black] (-2,0) -- (-3,1) -- (-3,-1) -- (-2,0);
    \filldraw[fill=black] (6,0) -- (7,1) -- (7,-1) -- (6,0);
    
    \begin{scope} [shift={(8,0)}]
        \draw (0,0) edge [->] (2,0);
    \end{scope}
    
    \begin{scope}[shift={(15,0)}]
    \node (F1) at (-4,0) {$T$};  
    \node (F2) at (2,1) {$B_C$};  
    \vertex [\MyRed, label=below:$x_1$] (a) at (0,-1) {};
    \vertex [\MyRed, label=below:$x_2$] (b) at (1,-1) {};
    \vertex [\MyRed, label=below:$x_3$] (c) at (2,-1) {};
    \vertex [\MyRed, label=below:$x_4$] (d) at (3,-1) {};
    \vertex [\MyRed, label=below:$x_5$] (e) at (4,-1) {};
    
    \vertex [\MyRed] (l1) at (0,0) {}; 
    \vertex [\MyRed] (l2) at (1,0) {};
    \vertex [\MyRed] (l3) at (2,0) {};
    \vertex [\MyRed] (l4) at (3,0) {};
    \vertex [\MyRed] (l5) at (4,0) {};
    
    \draw
    (l1) edge [\MyRed] (l5)
    (l1) edge [\MyRed] (a)
    (l2) edge [\MyRed] (b)
    (l3) edge [\MyRed] (c)
    (l4) edge [\MyRed] (d)
    (l5) edge [\MyRed] (e)
    (-2,0) edge node [label=below:$e_1$] {} (l1)
    (-2,0) edge [\MyRed] (l1)
    (l5) edge node [label=below:$e_5$] {} (6,0)
    (l5) edge (6,0)
    ;
    
    \filldraw[fill=black] (-2,0) -- (-3,1) -- (-3,-1) -- (-2,0);
    \filldraw[fill=black] (6,0) -- (7,1) -- (7,-1) -- (6,0);
    \end{scope}
    \end{tikzpicture}
    \caption{\stevenrevise{If $F_C$ has exactly one inside-out block $B$ (green) and $|B\cap C|\geq 3$ then $T[B]$ must use the edges $e_1$ and $e_r$ in the same
    way for every $T \in \T_C$. In this example $T[B]$ uses $e_1$ but not $e_5$, so this will hold for all $T \in \T_C$. This allows us to merge
    $B$ with the remainder of $C$ to obtain $B_C$ (red).}}
    \label{fig:one_insideout_size3ormore}
    \end{center}
\end{figure}
\vspace{0.4cm}
If $|B\cap C| = 2$ then we distinguish two subcases. In the first subcase, $T[B]$ always uses $e_1$ (but not $e_r$), always uses $e_r$ (but not $e_1$), or always uses both $e_1$ and $e_r$, ranging across all trees $T \in \T_C$. Here, as in the case $|B \cap C|=3$, merging $B$ with all other blocks that intersect $C$ yields the desired agreement forest. \stevenrevise{This is summarized in Figure \ref{fig:two_intersection_allsame}}. 

\begin{figure}[h]
    \begin{center}
    \begin{tikzpicture}[scale = 0.5]

    \node (F1) at (-4,0) {$T_1$};  
    \node (F2) at (0,1) {$B$};
    \vertex [label=below:$x_1$] (a) at (0,-1) {};
    \vertex [\MyGreen, label=below:$x_2$] (b) at (1,-1) {};
    \vertex [\MyGreen, label=below:$x_3$] (c) at (2,-1) {};
    \vertex [label=below:$x_4$] (d) at (3,-1) {};
    \vertex [label=below:$x_5$] (e) at (4,-1) {};
    
    \vertex [\MyGreen] (l1) at (0,0) {}; 
    \vertex [\MyGreen] (l2) at (1,0) {};
    \vertex [\MyGreen] (l3) at (2,0) {};
    \vertex (l4) at (3,0) {};
    \vertex (l5) at (4,0) {};
    
    \draw
    (l1) edge[\MyGreen] (l3)
    (l3) edge (l4)
    (l4) edge (l5)
    (l1) edge (a)
    (l2) edge [\MyGreen] (b)
    (l3) edge [\MyGreen] (c)
    (l4) edge (d)
    (l5) edge (e)
    (-2,0) edge node [label=below:$e_1$] {} (l1)
    (-2,0) edge[\MyGreen] (l1)
    (l5) edge node [label=below:$e_5$] {} (6,0)
    (l5) edge (6,0)
    ;
    
    \filldraw[fill=black] (-2,0) -- (-3,1) -- (-3,-1) -- (-2,0);
    \filldraw[fill=black] (6,0) -- (7,1) -- (7,-1) -- (6,0);
    
    \begin{scope} [shift={(0,-4)}]
    
    \node (F1) at (-4,0) {$T_2$};  
    \node (F2) at (4,1) {$B$};
    \vertex [label=below:$x_1$] (a) at (0,-1) {};
    \vertex [\MyGreen, label=below:$x_2$] (b) at (1,-1) {};
    \vertex [\MyGreen, label=below:$x_3$] (c) at (2,-1) {};
    \vertex [label=below:$x_4$] (d) at (3,-1) {};
    \vertex [label=below:$x_5$] (e) at (4,-1) {};
    
    \vertex (l1) at (0,0) {}; 
    \vertex [\MyGreen] (l2) at (1,0) {};
    \vertex [\MyGreen] (l3) at (2,0) {};
    \vertex [\MyGreen] (l4) at (3,0) {};
    \vertex [\MyGreen] (l5) at (4,0) {};
    
    \draw
    (l1) edge (l2)
    (l2) edge [\MyGreen] (l4)
    (l4) edge [\MyGreen] (l5)
    (l1) edge (a)
    (l2) edge [\MyGreen] (b)
    (l3) edge [\MyGreen] (c)
    (l4) edge (d)
    (l5) edge (e)
    (-2,0) edge node [label=below:$e_1$] {} (l1)
    (-2,0) edge (l1)
    (l5) edge node [label=below:$e_5$] {} (6,0)
    (l5) edge [\MyGreen] (6,0)
    ;
    
    \filldraw[fill=black] (-2,0) -- (-3,1) -- (-3,-1) -- (-2,0);
    \filldraw[fill=black] (6,0) -- (7,1) -- (7,-1) -- (6,0);
    \end{scope}

    \begin{scope} [shift={(0,-8)}]
    
    \node (F1) at (-4,0) {$T_3$};  
    \node (F2) at (2,1) {$B$};
    \vertex [label=below:$x_1$] (a) at (0,-1) {};
    \vertex [\MyGreen, label=below:$x_2$] (b) at (1,-1) {};
    \vertex [\MyGreen, label=below:$x_3$] (c) at (2,-1) {};
    \vertex [label=below:$x_4$] (d) at (3,-1) {};
    \vertex [label=below:$x_5$] (e) at (4,-1) {};
    
    \vertex [\MyGreen] (l1) at (0,0) {}; 
    \vertex [\MyGreen] (l2) at (1,0) {};
    \vertex [\MyGreen] (l3) at (2,0) {};
    \vertex [\MyGreen] (l4) at (3,0) {};
    \vertex [\MyGreen] (l5) at (4,0) {};
    
    \draw
    (l1) edge [\MyGreen] (l3)
    (l3) edge [\MyGreen] (l4)
    (l4) edge [\MyGreen] (l5)
    (l1) edge (a)
    (l2) edge [\MyGreen] (b)
    (l3) edge [\MyGreen] (c)
    (l4) edge (d)
    (l5) edge (e)
    (-2,0) edge node [label=below:$e_1$] {} (l1)
    (-2,0) edge[\MyGreen] (l1)
    (l5) edge node [label=below:$e_5$] {} (6,0)
    (l5) edge [\MyGreen] (6,0)
    ;
    
    \filldraw[fill=black] (-2,0) -- (-3,1) -- (-3,-1) -- (-2,0);
    \filldraw[fill=black] (6,0) -- (7,1) -- (7,-1) -- (6,0);
    \end{scope}
    
    \begin{scope} [shift={(8,0)}]
        \draw (0,0) edge [->] node  {} (2,0);
    \end{scope}

    \begin{scope} [shift={(8,-4)}]
        \draw (0,0) edge [->] node {} (2,0);
    \end{scope}

    \begin{scope} [shift={(8,-8)}]
        \draw (0,0) edge [->] node {} (2,0);
    \end{scope}
    
    \begin{scope}[shift={(15,0)}]
    \node (F1) at (-4,0) {$T_1$};  
    \node (F2) at (2,1) {$B_C$};  
    \vertex [\MyRed, label=below:$x_1$] (a) at (0,-1) {};
    \vertex [\MyRed, label=below:$x_2$] (b) at (1,-1) {};
    \vertex [\MyRed, label=below:$x_3$] (c) at (2,-1) {};
    \vertex [\MyRed, label=below:$x_4$] (d) at (3,-1) {};
    \vertex [\MyRed, label=below:$x_5$] (e) at (4,-1) {};
    
    \vertex [\MyRed] (l1) at (0,0) {}; 
    \vertex [\MyRed] (l2) at (1,0) {};
    \vertex [\MyRed] (l3) at (2,0) {};
    \vertex [\MyRed] (l4) at (3,0) {};
    \vertex [\MyRed] (l5) at (4,0) {};
    
    \draw
    (l1) edge [\MyRed] (l5)
    (l1) edge [\MyRed] (a)
    (l2) edge [\MyRed] (b)
    (l3) edge [\MyRed] (c)
    (l4) edge [\MyRed] (d)
    (l5) edge [\MyRed] (e)
    (-2,0) edge node [label=below:$e_1$] {} (l1)
    (-2,0) edge [\MyRed] (l1)
    (l5) edge node [label=below:$e_5$] {} (6,0)
    (l5) edge (6,0)
    ;
    
    \filldraw[fill=black] (-2,0) -- (-3,1) -- (-3,-1) -- (-2,0);
    \filldraw[fill=black] (6,0) -- (7,1) -- (7,-1) -- (6,0);
    \end{scope}

    \begin{scope}[shift={(15,-4)}]
    \node (F1) at (-4,0) {$T_2$};  
    \node (F2) at (2,1) {$B_C$};  
    \vertex [\MyRed, label=below:$x_1$] (a) at (0,-1) {};
    \vertex [\MyRed, label=below:$x_2$] (b) at (1,-1) {};
    \vertex [\MyRed, label=below:$x_3$] (c) at (2,-1) {};
    \vertex [\MyRed, label=below:$x_4$] (d) at (3,-1) {};
    \vertex [\MyRed, label=below:$x_5$] (e) at (4,-1) {};
    
    \vertex [\MyRed] (l1) at (0,0) {}; 
    \vertex [\MyRed] (l2) at (1,0) {};
    \vertex [\MyRed] (l3) at (2,0) {};
    \vertex [\MyRed] (l4) at (3,0) {};
    \vertex [\MyRed] (l5) at (4,0) {};
    
    \draw
    (l1) edge [\MyRed] (l5)
    (l1) edge [\MyRed] (a)
    (l2) edge [\MyRed] (b)
    (l3) edge [\MyRed] (c)
    (l4) edge [\MyRed] (d)
    (l5) edge [\MyRed] (e)
    (-2,0) edge node [label=below:$e_1$] {} (l1)
    (-2,0) edge (l1)
    (l5) edge node [label=below:$e_5$] {} (6,0)
    (l5) edge [\MyRed](6,0)
    ;
    
    \filldraw[fill=black] (-2,0) -- (-3,1) -- (-3,-1) -- (-2,0);
    \filldraw[fill=black] (6,0) -- (7,1) -- (7,-1) -- (6,0);
    \end{scope}

    \begin{scope}[shift={(15,-8)}]
    \node (F1) at (-4,0) {$T_3$};  
    \node (F2) at (2,1) {$B_C$};  
    \vertex [\MyRed, label=below:$x_1$] (a) at (0,-1) {};
    \vertex [\MyRed, label=below:$x_2$] (b) at (1,-1) {};
    \vertex [\MyRed, label=below:$x_3$] (c) at (2,-1) {};
    \vertex [\MyRed, label=below:$x_4$] (d) at (3,-1) {};
    \vertex [\MyRed, label=below:$x_5$] (e) at (4,-1) {};
    
    \vertex [\MyRed] (l1) at (0,0) {}; 
    \vertex [\MyRed] (l2) at (1,0) {};
    \vertex [\MyRed] (l3) at (2,0) {};
    \vertex [\MyRed] (l4) at (3,0) {};
    \vertex [\MyRed] (l5) at (4,0) {};
    
    \draw
    (l1) edge [\MyRed] (l5)
    (l1) edge [\MyRed] (a)
    (l2) edge [\MyRed] (b)
    (l3) edge [\MyRed] (c)
    (l4) edge [\MyRed] (d)
    (l5) edge [\MyRed] (e)
    (-2,0) edge node [label=below:$e_1$] {} (l1)
    (-2,0) edge [\MyRed](l1)
    (l5) edge node [label=below:$e_5$] {} (6,0)
    (l5) edge [\MyRed](6,0)
    ;
    
    \filldraw[fill=black] (-2,0) -- (-3,1) -- (-3,-1) -- (-2,0);
    \filldraw[fill=black] (6,0) -- (7,1) -- (7,-1) -- (6,0);
    \end{scope}
    \end{tikzpicture}
    \caption{\stevenrevise{If $F_C$ has exactly one inside-out block $B$ (green) and $|B\cap C| = 2$, and $T[B]$ uses the edges $e_1$ and $e_r$ in the same way for every $T \in \T_C$ - so in every tree it behaves like in $T_1$, or in every tree it behaves like in $T_2$, or in every tree it behaves like in $T_3$ - then we can merge
    $B$ with the remainder of $C$ to obtain $B_C$ (red).}}
    \label{fig:two_intersection_allsame}
    \end{center}
\end{figure}
\vspace{0.4cm}
In the second subcase there are distinct trees $T, T' \in \T_C$ where $T[B]$ uses $e_1$ but not $e_r$, while $T'[B]$ uses $e_r$ but not $e_1$\footnote{If $|B \cap C|=2$ and some tree $T$ has the property that $T[B]$ uses both $e_1$ and $e_r$, then all trees $T \in \T_C$ have this property.}.
Combined with the fact that $r \geq 3$, this means that $|C \setminus B| \geq 1$ and all taxa in $C \setminus B$ are singleton blocks in the agreement forest. We split $B$ into $B \setminus C$ and $B \cap C$ and then merge $B \cap C$ with the singleton blocks in $C \setminus B$. This does not increase the size of the agreement forest, so we are done with the $|B \cap C|=2$ case. \stevenrevise{This is depicted in Figure \ref{fig:two_intersection_different}}.

\begin{figure}[h]
    \begin{center}
    \begin{tikzpicture}[scale = 0.5]

    \node (F1) at (-4,0) {$T_1$};  
    \node (F2) at (0,1) {$B$};
    \vertex [\MyOrange, label=below:$x_1$] (a) at (0,-1) {};
    \vertex [\MyGreen, label=below:$x_2$] (b) at (1,-1) {};
    \vertex [\MyGreen, label=below:$x_3$] (c) at (2,-1) {};
    \vertex [label=below:$x_4$] (d) at (3,-1) {};
    \vertex [label=below:$x_5$] (e) at (4,-1) {};
    
    \vertex [\MyGreen] (l1) at (0,0) {}; 
    \vertex [\MyGreen] (l2) at (1,0) {};
    \vertex [\MyGreen] (l3) at (2,0) {};
    \vertex (l4) at (3,0) {};
    \vertex (l5) at (4,0) {};
    
    \draw
    (l1) edge[\MyGreen] (l3)
    (l3) edge (l4)
    (l4) edge (l5)
    (l1) edge (a)
    (l2) edge [\MyGreen] (b)
    (l3) edge [\MyGreen] (c)
    (l4) edge (d)
    (l5) edge (e)
    (-2,0) edge node [label=below:$e_1$] {} (l1)
    (-2,0) edge[\MyGreen] (l1)
    (l5) edge node [label=below:$e_5$] {} (6,0)
    (l5) edge (6,0)
    ;
    
    \filldraw[fill=black] (-2,0) -- (-3,1) -- (-3,-1) -- (-2,0);
    \filldraw[fill=black] (6,0) -- (7,1) -- (7,-1) -- (6,0);
    
    \begin{scope} [shift={(0,-4)}]
    
    \node (F1) at (-4,0) {$T_2$};  
    \node (F2) at (4,1) {$B$};
    \vertex [label=below:$x_1$] (a) at (0,-1) {};
    \vertex [\MyGreen, label=below:$x_2$] (b) at (1,-1) {};
    \vertex [\MyGreen, label=below:$x_3$] (c) at (2,-1) {};
    \vertex [\MyOrange, label=below:$x_4$] (d) at (3,-1) {};
    \vertex [\MyOrange, label=below:$x_5$] (e) at (4,-1) {};
    
    \vertex (l1) at (0,0) {}; 
    \vertex [\MyGreen] (l2) at (1,0) {};
    \vertex [\MyGreen] (l3) at (2,0) {};
    \vertex [\MyGreen] (l4) at (3,0) {};
    \vertex [\MyGreen] (l5) at (4,0) {};
    
    \draw
    (l1) edge (l2)
    (l2) edge [\MyGreen] (l4)
    (l4) edge [\MyGreen] (l5)
    (l1) edge (a)
    (l2) edge [\MyGreen] (b)
    (l3) edge [\MyGreen] (c)
    (l4) edge (d)
    (l5) edge (e)
    (-2,0) edge node [label=below:$e_1$] {} (l1)
    (-2,0) edge (l1)
    (l5) edge node [label=below:$e_5$] {} (6,0)
    (l5) edge [\MyGreen] (6,0)
    ;
    
    \filldraw[fill=black] (-2,0) -- (-3,1) -- (-3,-1) -- (-2,0);
    \filldraw[fill=black] (6,0) -- (7,1) -- (7,-1) -- (6,0);
    \end{scope}

    \begin{scope} [shift={(8,0)}]
        \draw (0,0) edge [->] node  {} (2,0);
    \end{scope}

    \begin{scope} [shift={(8,-4)}]
        \draw (0,0) edge [->] node  {} (2,0);
    \end{scope}

    \begin{scope}[shift={(15,0)}]
    \node (F1) at (-4,0) {$T_1$};  
    \node (F2) at (2,1) {$B_C$};  
    \vertex [\MyRed, label=below:$x_1$] (a) at (0,-1) {};
    \vertex [\MyRed, label=below:$x_2$] (b) at (1,-1) {};
    \vertex [\MyRed, label=below:$x_3$] (c) at (2,-1) {};
    \vertex [\MyRed, label=below:$x_4$] (d) at (3,-1) {};
    \vertex [\MyRed, label=below:$x_5$] (e) at (4,-1) {};
    
    \vertex [\MyRed] (l1) at (0,0) {}; 
    \vertex [\MyRed] (l2) at (1,0) {};
    \vertex [\MyRed] (l3) at (2,0) {};
    \vertex [\MyRed] (l4) at (3,0) {};
    \vertex [\MyRed] (l5) at (4,0) {};
    
    \draw
    (l1) edge [\MyRed] (l5)
    (l1) edge [\MyRed] (a)
    (l2) edge [\MyRed] (b)
    (l3) edge [\MyRed] (c)
    (l4) edge [\MyRed] (d)
    (l5) edge [\MyRed] (e)
    (-2,0) edge node [label=below:$e_1$] {} (l1)
    (-2,0) edge (l1)
    (l5) edge node [label=below:$e_5$] {} (6,0)
    (l5) edge (6,0)
    ;
    
    \filldraw[fill=black] (-2,0) -- (-3,1) -- (-3,-1) -- (-2,0);
    \filldraw[fill=black] (6,0) -- (7,1) -- (7,-1) -- (6,0);
    \end{scope}

    \begin{scope}[shift={(15,-4)}]
    \node (F1) at (-4,0) {$T_2$};  
    \node (F2) at (2,1) {$B_C$};  
    \vertex [\MyRed, label=below:$x_1$] (a) at (0,-1) {};
    \vertex [\MyRed, label=below:$x_2$] (b) at (1,-1) {};
    \vertex [\MyRed, label=below:$x_3$] (c) at (2,-1) {};
    \vertex [\MyRed, label=below:$x_4$] (d) at (3,-1) {};
    \vertex [\MyRed, label=below:$x_5$] (e) at (4,-1) {};
    
    \vertex [\MyRed] (l1) at (0,0) {}; 
    \vertex [\MyRed] (l2) at (1,0) {};
    \vertex [\MyRed] (l3) at (2,0) {};
    \vertex [\MyRed] (l4) at (3,0) {};
    \vertex [\MyRed] (l5) at (4,0) {};
    
    \draw
    (l1) edge [\MyRed] (l5)
    (l1) edge [\MyRed] (a)
    (l2) edge [\MyRed] (b)
    (l3) edge [\MyRed] (c)
    (l4) edge [\MyRed] (d)
    (l5) edge [\MyRed] (e)
    (-2,0) edge node [label=below:$e_1$] {} (l1)
    (-2,0) edge (l1)
    (l5) edge node [label=below:$e_5$] {} (6,0)
    (l5) edge (6,0)
    ;
    
    \filldraw[fill=black] (-2,0) -- (-3,1) -- (-3,-1) -- (-2,0);
    \filldraw[fill=black] (6,0) -- (7,1) -- (7,-1) -- (6,0);
    \end{scope}
    \end{tikzpicture}
    \caption{\stevenrevise{If $F_C$ has exactly one inside-out block $B$ (green) and $|B\cap C| = 2$, and $T[B]$ uses the edges $e_1$ and $e_r$ in different ways - so in at least one tree it behaves like in $T_1$, and in at least one tree it behaves like in $T_2$ - then all the taxa in $C \setminus (B \cap C)$ are necessarily singletons in $F_C$ (shown in orange). In this case we can split $B \cap C$ off from $B$ and merge it with these singletons to obtain $B_C$ (red).}}
    \label{fig:two_intersection_different}
    \end{center}
\end{figure}
\vspace{0.4cm}

    If $|B \cap C|=1$ we proceed as follows. Consider an arbitrary tree $T\in\T_C$. Then $T[B]$ uses one or both of the edges $e_1,e_r$, and this can vary for different $T \in \T_C$. 
    Observe that due to $r \geq 3$ and $|B \cap C|=1$ there is at least one block in the agreement forest distinct from $B$ that intersects $C$.

    If $T[B]$ uses $e_1$ (but not $e_r$) for all $T \in \T_C$, or $T[B]$ uses $e_r$ (but not $e_1$) for all $T \in \T_C$, we can simply merge $B$ with all other blocks that intersect $C$ and we are done; \stevenrevise{see Figure \ref{fig:one_intersection_Same}}.

\begin{figure}[h]
    \begin{center}
    \begin{tikzpicture}[scale = 0.5]

    \node (F1) at (-4,0) {$T_1$};  
    \node (F2) at (0,1) {$B$};
    \vertex [label=below:$x_1$] (a) at (0,-1) {};
    \vertex [label=below:$x_2$] (b) at (1,-1) {};
    \vertex [\MyGreen, label=below:$x_3$] (c) at (2,-1) {};
    \vertex [label=below:$x_4$] (d) at (3,-1) {};
    \vertex [label=below:$x_5$] (e) at (4,-1) {};
    
    \vertex [\MyGreen] (l1) at (0,0) {}; 
    \vertex [\MyGreen] (l2) at (1,0) {};
    \vertex [\MyGreen] (l3) at (2,0) {};
    \vertex (l4) at (3,0) {};
    \vertex (l5) at (4,0) {};
    
    \draw
    (l1) edge[\MyGreen] (l3)
    (l3) edge (l4)
    (l4) edge (l5)
    (l1) edge (a)
    (l2) edge (b)
    (l3) edge [\MyGreen] (c)
    (l4) edge (d)
    (l5) edge (e)
    (-2,0) edge node [label=below:$e_1$] {} (l1)
    (-2,0) edge[\MyGreen] (l1)
    (l5) edge node [label=below:$e_5$] {} (6,0)
    (l5) edge (6,0)
    ;
    
    \filldraw[fill=black] (-2,0) -- (-3,1) -- (-3,-1) -- (-2,0);
    \filldraw[fill=black] (6,0) -- (7,1) -- (7,-1) -- (6,0);
    
    \begin{scope} [shift={(0,-4)}]
    
    \node (F1) at (-4,0) {$T_2$};  
    \node (F2) at (4,1) {$B$};
    \vertex [label=below:$x_1$] (a) at (0,-1) {};
    \vertex [label=below:$x_2$] (b) at (1,-1) {};
    \vertex [\MyGreen, label=below:$x_3$] (c) at (2,-1) {};
    \vertex [label=below:$x_4$] (d) at (3,-1) {};
    \vertex [label=below:$x_5$] (e) at (4,-1) {};
    
    \vertex (l1) at (0,0) {}; 
    \vertex (l2) at (1,0) {};
    \vertex [\MyGreen] (l3) at (2,0) {};
    \vertex [\MyGreen] (l4) at (3,0) {};
    \vertex [\MyGreen] (l5) at (4,0) {};
    
    \draw
    (l1) edge (l2)
    (l2) edge (l3)
    (l3) edge [\MyGreen] (l4)
    (l4) edge [\MyGreen] (l5)
    (l1) edge (a)
    (l2) edge (b)
    (l3) edge [\MyGreen] (c)
    (l4) edge (d)
    (l5) edge (e)
    (-2,0) edge node [label=below:$e_1$] {} (l1)
    (-2,0) edge (l1)
    (l5) edge node [label=below:$e_5$] {} (6,0)
    (l5) edge [\MyGreen](6,0)
    ;

    \filldraw[fill=black] (-2,0) -- (-3,1) -- (-3,-1) -- (-2,0);
    \filldraw[fill=black] (6,0) -- (7,1) -- (7,-1) -- (6,0);
    \end{scope}
    
    \begin{scope} [shift={(8,0)}]
        \draw (0,0) edge [->] node  {} (2,0);
    \end{scope}

    \begin{scope} [shift={(8,-4)}]
        \draw (0,0) edge [->] node  {} (2,0);
    \end{scope}
    
    \begin{scope}[shift={(15,0)}]
    \node (F1) at (-4,0) {$T_1$};  
    \node (F2) at (2,1) {$B_C$};  
    \vertex [\MyRed, label=below:$x_1$] (a) at (0,-1) {};
    \vertex [\MyRed, label=below:$x_2$] (b) at (1,-1) {};
    \vertex [\MyRed, label=below:$x_3$] (c) at (2,-1) {};
    \vertex [\MyRed, label=below:$x_4$] (d) at (3,-1) {};
    \vertex [\MyRed, label=below:$x_5$] (e) at (4,-1) {};
    
    \vertex [\MyRed] (l1) at (0,0) {}; 
    \vertex [\MyRed] (l2) at (1,0) {};
    \vertex [\MyRed] (l3) at (2,0) {};
    \vertex [\MyRed] (l4) at (3,0) {};
    \vertex [\MyRed] (l5) at (4,0) {};
    
    \draw
    (l1) edge [\MyRed] (l5)
    (l1) edge [\MyRed] (a)
    (l2) edge [\MyRed] (b)
    (l3) edge [\MyRed] (c)
    (l4) edge [\MyRed] (d)
    (l5) edge [\MyRed] (e)
    (-2,0) edge node [label=below:$e_1$] {} (l1)
    (-2,0) edge [\MyRed] (l1)
    (l5) edge node [label=below:$e_5$] {} (6,0)
    (l5) edge (6,0)
    ;
    
    \filldraw[fill=black] (-2,0) -- (-3,1) -- (-3,-1) -- (-2,0);
    \filldraw[fill=black] (6,0) -- (7,1) -- (7,-1) -- (6,0);
    \end{scope}

    \begin{scope}[shift={(15,-4)}]
    \node (F1) at (-4,0) {$T_2$};  
    \node (F2) at (2,1) {$B_C$};  
    \vertex [\MyRed, label=below:$x_1$] (a) at (0,-1) {};
    \vertex [\MyRed, label=below:$x_2$] (b) at (1,-1) {};
    \vertex [\MyRed, label=below:$x_3$] (c) at (2,-1) {};
    \vertex [\MyRed, label=below:$x_4$] (d) at (3,-1) {};
    \vertex [\MyRed, label=below:$x_5$] (e) at (4,-1) {};
    
    \vertex [\MyRed] (l1) at (0,0) {}; 
    \vertex [\MyRed] (l2) at (1,0) {};
    \vertex [\MyRed] (l3) at (2,0) {};
    \vertex [\MyRed] (l4) at (3,0) {};
    \vertex [\MyRed] (l5) at (4,0) {};
    
    \draw
    (l1) edge [\MyRed] (l5)
    (l1) edge [\MyRed] (a)
    (l2) edge [\MyRed] (b)
    (l3) edge [\MyRed] (c)
    (l4) edge [\MyRed] (d)
    (l5) edge [\MyRed] (e)
    (-2,0) edge node [label=below:$e_1$] {} (l1)
    (-2,0) edge (l1)
    (l5) edge node [label=below:$e_5$] {} (6,0)
    (l5) edge [\MyRed](6,0)
    ;
    
    \filldraw[fill=black] (-2,0) -- (-3,1) -- (-3,-1) -- (-2,0);
    \filldraw[fill=black] (6,0) -- (7,1) -- (7,-1) -- (6,0);
    \end{scope}

    \end{tikzpicture}
    \caption{\stevenrevise{If $F_C$ has exactly one inside-out block $B$ (green) and $|B\cap C| = 1$, and $T[B]$ always behaves like in $T_1$, or always behaves like in $T_2$, then we can merge $B$ with the remainder of $C$ to obtain $B_C$ (red).}}
    \label{fig:one_intersection_Same}
    \end{center}
\end{figure}
\vspace{0.4cm}

    If there exist distinct trees $T, T' \in \T_C$ such that $T[B]$ uses $e_1$ (but not $e_r$) and $T'[B]$ uses $e_r$ (but not $e_1$),   
    then the at least two taxa in $C \setminus B$ are both singleton taxa. Now, select an arbitrary $T \in \T_C$ and let $x$ be the single taxon in $|B \cap C|$. Let $L$ and $R$ be the taxa in the two sibling subtrees of $x$ in $T|B$. (If $|B|=2$ then only one of $L$ and $R$ will be non-empty, say $L$, otherwise they will both be non-empty). \stevenrevise{Figure \ref{fig:siblings} clarifies this concept.}

\begin{figure}[h]
    \begin{center}
    \begin{tikzpicture}[scale = 0.5]

    \node (F1) at (-5,0) {$T|B$};
    \node (F1) at (-2,1.5) {$L$};
    \node (F1) at (6,1.5) {$R$};
    \vertex [\MyGreen, label=below:$x$] (c) at (2,-2) {};
    
    \vertex [\MyGreen] (l3) at (2,0) {};
    
    \draw
    (l3) edge node  {} (c)
    (l3) edge [\MyGreen] (c)
    (-2,0) edge node {} (l3)
    (-2,0) edge[\MyGreen] (l3)
    (l3) edge node  {} (6,0)
    (l3) edge[\MyGreen] (6,0)
    ;
    
    \filldraw[fill=\MyGreen] (-2,0) -- (-3,1) -- (-3,-1) -- (-2,0);
    \filldraw[fill=\MyGreen] (6,0) -- (7,1) -- (7,-1) -- (6,0);
    \end{tikzpicture}
    \caption{\stevenrevise{The definition of the sets $L$ and $R$ as used in Figure \ref{fig:one_intersection_two_of_three.} (where $x = x_3$). Here $T$ can be any tree in $\T_C$.}}
        \label{fig:siblings}
    \end{center}
\end{figure}

    Now, we split $B$ into $\{x\}$, $L$ and (if it exists) $R$ but compensate for this by merging $\{x\}$ with the at least two singleton blocks corresponding to taxa in $C \setminus B$. This does not increase the size of the agreement forest overall, so we are done. \stevenrevise{See Figure \ref{fig:one_intersection_two_of_three.}}.

Suppose then that there exists a tree $T \in \T_C$ such that $T[B]$ uses both $e_1$ and $e_r$. The same split-and-merge tactic with $L$ and $R$ as used in the previous paragraph also works here. This concludes the case that $|B \cap C|=1$. \stevenrevise{This is also depicted in Figure \ref{fig:one_intersection_two_of_three.}}.

\begin{figure}[h]
    \begin{center}
    \begin{tikzpicture}[scale = 0.5]

    \node (F1) at (-4,0) {$T_1$};  
    \node (F2) at (0,1) {$B$};
    \vertex [\MyOrange, label=below:$x_1$] (a) at (0,-1) {};
    \vertex [\MyOrange, label=below:$x_2$] (b) at (1,-1) {};
    \vertex [\MyGreen, label=below:$x_3$] (c) at (2,-1) {};
    \vertex [label=below:$x_4$] (d) at (3,-1) {};
    \vertex [label=below:$x_5$] (e) at (4,-1) {};
    
    \vertex [\MyGreen] (l1) at (0,0) {}; 
    \vertex [\MyGreen] (l2) at (1,0) {};
    \vertex [\MyGreen] (l3) at (2,0) {};
    \vertex (l4) at (3,0) {};
    \vertex (l5) at (4,0) {};
    
    \draw
    (l1) edge[\MyGreen] (l3)
    (l3) edge (l4)
    (l4) edge (l5)
    (l1) edge (a)
    (l2) edge (b)
    (l3) edge [\MyGreen] (c)
    (l4) edge (d)
    (l5) edge (e)
    (-2,0) edge node [label=below:$e_1$] {} (l1)
    (-2,0) edge[\MyGreen] (l1)
    (l5) edge node [label=below:$e_5$] {} (6,0)
    (l5) edge (6,0)
    ;
    
    \filldraw[fill=black] (-2,0) -- (-3,1) -- (-3,-1) -- (-2,0);
    \filldraw[fill=black] (6,0) -- (7,1) -- (7,-1) -- (6,0);
    
    \begin{scope} [shift={(0,-4)}]
    
    \node (F1) at (-4,0) {$T_2$};  
    \node (F2) at (4,1) {$B$};
    \vertex [label=below:$x_1$] (a) at (0,-1) {};
    \vertex [label=below:$x_2$] (b) at (1,-1) {};
    \vertex [\MyGreen, label=below:$x_3$] (c) at (2,-1) {};
    \vertex [\MyOrange, label=below:$x_4$] (d) at (3,-1) {};
    \vertex [\MyOrange, label=below:$x_5$] (e) at (4,-1) {};
    
    \vertex (l1) at (0,0) {}; 
    \vertex (l2) at (1,0) {};
    \vertex [\MyGreen] (l3) at (2,0) {};
    \vertex [\MyGreen] (l4) at (3,0) {};
    \vertex [\MyGreen] (l5) at (4,0) {};
    
    \draw
    (l1) edge (l2)
    (l2) edge (l3)
    (l3) edge [\MyGreen] (l4)
    (l4) edge [\MyGreen] (l5)
    (l1) edge (a)
    (l2) edge (b)
    (l3) edge [\MyGreen] (c)
    (l4) edge (d)
    (l5) edge (e)
    (-2,0) edge node [label=below:$e_1$] {} (l1)
    (-2,0) edge (l1)
    (l5) edge node [label=below:$e_5$] {} (6,0)
    (l5) edge [\MyGreen] (6,0)
    ;
    
    \filldraw[fill=black] (-2,0) -- (-3,1) -- (-3,-1) -- (-2,0);
    \filldraw[fill=black] (6,0) -- (7,1) -- (7,-1) -- (6,0);
    \end{scope}

    \begin{scope} [shift={(0,-8)}]
    
    \node (F1) at (-4,0) {$T_3$};  
    \node (F2) at (2,1) {$B$};
    \vertex [\MyOrange, label=below:$x_1$] (a) at (0,-1) {};
    \vertex [\MyOrange, label=below:$x_2$] (b) at (1,-1) {};
    \vertex [\MyGreen, label=below:$x_3$] (c) at (2,-1) {};
    \vertex [\MyOrange, label=below:$x_4$] (d) at (3,-1) {};
    \vertex [\MyOrange, label=below:$x_5$] (e) at (4,-1) {};
    
    \vertex [\MyGreen] (l1) at (0,0) {}; 
    \vertex [\MyGreen] (l2) at (1,0) {};
    \vertex [\MyGreen] (l3) at (2,0) {};
    \vertex [\MyGreen] (l4) at (3,0) {};
    \vertex [\MyGreen] (l5) at (4,0) {};
    
    \draw
    (l1) edge [\MyGreen] (l2)
    (l2) edge [\MyGreen] (l4)
    (l4) edge [\MyGreen] (l5)
    (l1) edge (a)
    (l2) edge (b)
    (l3) edge [\MyGreen] (c)
    (l4) edge (d)
    (l5) edge (e)
    (-2,0) edge node [label=below:$e_1$] {} (l1)
    (-2,0) edge [\MyGreen] (l1)
    (l5) edge node [label=below:$e_5$] {} (6,0)
    (l5) edge [\MyGreen] (6,0)
    ;
    
    \filldraw[fill=black] (-2,0) -- (-3,1) -- (-3,-1) -- (-2,0);
    \filldraw[fill=black] (6,0) -- (7,1) -- (7,-1) -- (6,0);
    \end{scope}
    
    \begin{scope} [shift={(8,-4)}]
        \draw (0,0) edge [->] (2,0);
    \end{scope}
    
    \begin{scope}[shift={(15,0)}]
    \node (F1) at (-4,0) {$T_1$};  
    \node (F2) at (2,1) {$B_C$};  
    \vertex [\MyRed, label=below:$x_1$] (a) at (0,-1) {};
    \vertex [\MyRed, label=below:$x_2$] (b) at (1,-1) {};
    \vertex [\MyRed, label=below:$x_3$] (c) at (2,-1) {};
    \vertex [\MyRed, label=below:$x_4$] (d) at (3,-1) {};
    \vertex [\MyRed, label=below:$x_5$] (e) at (4,-1) {};
    
    \vertex [\MyRed] (l1) at (0,0) {}; 
    \vertex [\MyRed] (l2) at (1,0) {};
    \vertex [\MyRed] (l3) at (2,0) {};
    \vertex [\MyRed] (l4) at (3,0) {};
    \vertex [\MyRed] (l5) at (4,0) {};
    
    \draw
    (l1) edge [\MyRed] (l5)
    (l1) edge [\MyRed] (a)
    (l2) edge [\MyRed] (b)
    (l3) edge [\MyRed] (c)
    (l4) edge [\MyRed] (d)
    (l5) edge [\MyRed] (e)
    (-2,0) edge node [label=below:$e_1$] {} (l1)
    (-2,0) edge (l1)
    (l5) edge node [label=below:$e_5$] {} (6,0)
    (l5) edge (6,0)
    ;
    
    \filldraw[fill=black] (-2,0) -- (-3,1) -- (-3,-1) -- (-2,0);
    \filldraw[fill=black] (6,0) -- (7,1) -- (7,-1) -- (6,0);
    \end{scope}

    \begin{scope}[shift={(15,-4)}]
    \node (F1) at (-4,0) {$T_2$};  
    \node (F2) at (2,1) {$B_C$};  
    \vertex [\MyRed, label=below:$x_1$] (a) at (0,-1) {};
    \vertex [\MyRed, label=below:$x_2$] (b) at (1,-1) {};
    \vertex [\MyRed, label=below:$x_3$] (c) at (2,-1) {};
    \vertex [\MyRed, label=below:$x_4$] (d) at (3,-1) {};
    \vertex [\MyRed, label=below:$x_5$] (e) at (4,-1) {};
    
    \vertex [\MyRed] (l1) at (0,0) {}; 
    \vertex [\MyRed] (l2) at (1,0) {};
    \vertex [\MyRed] (l3) at (2,0) {};
    \vertex [\MyRed] (l4) at (3,0) {};
    \vertex [\MyRed] (l5) at (4,0) {};
    
    \draw
    (l1) edge [\MyRed] (l5)
    (l1) edge [\MyRed] (a)
    (l2) edge [\MyRed] (b)
    (l3) edge [\MyRed] (c)
    (l4) edge [\MyRed] (d)
    (l5) edge [\MyRed] (e)
    (-2,0) edge node [label=below:$e_1$] {} (l1)
    (-2,0) edge (l1)
    (l5) edge node [label=below:$e_5$] {} (6,0)
    (l5) edge (6,0)
    ;
    
    \filldraw[fill=black] (-2,0) -- (-3,1) -- (-3,-1) -- (-2,0);
    \filldraw[fill=black] (6,0) -- (7,1) -- (7,-1) -- (6,0);
    \end{scope}

    \begin{scope}[shift={(15,-8)}]
    \node (F1) at (-4,0) {$T_3$};  
    \node (F2) at (2,1) {$B_C$};  
    \vertex [\MyRed, label=below:$x_1$] (a) at (0,-1) {};
    \vertex [\MyRed, label=below:$x_2$] (b) at (1,-1) {};
    \vertex [\MyRed, label=below:$x_3$] (c) at (2,-1) {};
    \vertex [\MyRed, label=below:$x_4$] (d) at (3,-1) {};
    \vertex [\MyRed, label=below:$x_5$] (e) at (4,-1) {};
    
    \vertex [\MyRed] (l1) at (0,0) {}; 
    \vertex [\MyRed] (l2) at (1,0) {};
    \vertex [\MyRed] (l3) at (2,0) {};
    \vertex [\MyRed] (l4) at (3,0) {};
    \vertex [\MyRed] (l5) at (4,0) {};
    
    \draw
    (l1) edge [\MyRed] (l5)
    (l1) edge [\MyRed] (a)
    (l2) edge [\MyRed] (b)
    (l3) edge [\MyRed] (c)
    (l4) edge [\MyRed] (d)
    (l5) edge [\MyRed] (e)
    (-2,0) edge node [label=below:$e_1$] {} (l1)
    (-2,0) edge (l1)
    (l5) edge node [label=below:$e_5$] {} (6,0)
    (l5) edge (6,0)
    ;
    
    \filldraw[fill=black] (-2,0) -- (-3,1) -- (-3,-1) -- (-2,0);
    \filldraw[fill=black] (6,0) -- (7,1) -- (7,-1) -- (6,0);
    \end{scope}

    \end{tikzpicture}
    \caption{\stevenrevise{If $F_C$ has exactly one inside-out block $B$ (green) and $|B\cap C| = 1$, and at least one of the two following situations holds -- (i) in at least one tree $T[B]$ behaves like in $T_1$ and in at least one tree $T[B]$ behaves like in $T_2$, (ii) in at least one tree
    $T[B]$ behaves like in $T_3$ -- then all the taxa in $(C \setminus B)$ are necessarily singletons (shown in orange). After splitting $B$ into $L$, $R$ and $(B \cap C)$ (see Figure \ref{fig:siblings} for the definition of $L$ and $R$) all the blocks intersecting $C$ can be merged into a single block $B_C$ (red).
}}
  \label{fig:one_intersection_two_of_three.}
    \end{center}
\end{figure}
\vspace{0.4cm}

    The third case is that $\Fu_C$ has \textbf{no inside-out blocks}. We say that $B$ is a \emph{bypass block} with respect to $T \in \T_C$ if $B \cap C = \emptyset$ but $T[B]$ uses both $e_1$ and $e_r$. If $\Fu_C$ has no bypass blocks with respect to any tree $T \in \T_C$, then we can construct $\Fu_C'$ from $\Fu_C$ simply by merging all blocks contained in $C$ into a single block $C$ and we are done.
    
    Now assume that $\Fu_C$ has at least one bypass block $B$ with respect to some $T \in \T_C$. This implies that all elements of $C$ are singleton blocks in $\Fu_C$, so $\Fu_C$ has size at least $r+1$. Since $\Fu_C$ has size at most $\ku$ by assumption, it follows that $3 \leq r < \ku$ and hence that $r=t+1$. In this case, we construct $\Fu_C'$ from $\Fu_C$ as follows. While there exists a tree $T\in\T_C$ for which $B$ is a bypass block with respect to $T$, split $B$ 
    into~$B_L,B_R$ where $T[B_L]$ is incident to edge $e_1$ \stevenrevise{and} $T[B_R]$ is incident to edge $e_r$.

    After repeating this for all trees for which there is a bypass block (possibly splitting the same block multiple times), merge all blocks contained in $C$ (which are all singletons) into a single block $C$. \stevenrevise{See Figure \ref{fig:bypass_split}}. 
    It is again easy to see that~$\Fu_C'$ is an unrooted agreement forest. Clearly there is not overlap because there are no bypass blocks any more. Isomorphism again holds because $C$ has the exact same topology in every tree and we only split bypasses.

\begin{figure}[h]
    \begin{center}
    \begin{tikzpicture}[scale = 0.5]

    \node (F1) at (-4,0) {$T_1$};  
    \node (F2) at (2,1) {$B$};
    \vertex [\MyOrange, label=below:$x_1$] (a) at (0,-1) {};
    \vertex [\MyOrange, label=below:$x_2$] (b) at (1,-1) {};
    \vertex [\MyOrange, label=below:$x_3$] (c) at (2,-1) {};
    \vertex [\MyOrange, label=below:$x_4$] (d) at (3,-1) {};
    \vertex [\MyOrange, label=below:$x_5$] (e) at (4,-1) {};
    
    \vertex [\MyGreen] (l1) at (0,0) {}; 
    \vertex [\MyGreen] (l2) at (1,0) {};
    \vertex [\MyGreen] (l3) at (2,0) {};
    \vertex [\MyGreen] (l4) at (3,0) {};
    \vertex [\MyGreen] (l5) at (4,0) {};
    
    \draw
    (l1) edge [\MyGreen] (l5)
    (l1) edge (a)
    (l2) edge (b)
    (l3) edge (c)
    (l4) edge (d)
    (l5) edge (e)
    (-2,0) edge node [label=below:$e_1$] {} (l1)
    (-2,0) edge [\MyGreen] (l1)
    (l5) edge node [label=below:$e_5$] {} (6,0)
    (l5) edge [\MyGreen](6,0)
    ;
    
    \filldraw[fill=black] (-2,0) -- (-3,1) -- (-3,-1) -- (-2,0);
    \filldraw[fill=black] (6,0) -- (7,1) -- (7,-1) -- (6,0);

    \begin{scope}[shift={(0,-4)}]
    \node (F1) at (-4,0) {$T_2$};  
    \node (F2) at (2,1) {$B'$};
    \vertex [\MyOrange, label=below:$x_1$] (a) at (0,-1) {};
    \vertex [\MyOrange, label=below:$x_2$] (b) at (1,-1) {};
    \vertex [\MyOrange, label=below:$x_3$] (c) at (2,-1) {};
    \vertex [\MyOrange, label=below:$x_4$] (d) at (3,-1) {};
    \vertex [\MyOrange, label=below:$x_5$] (e) at (4,-1) {};
    
    \vertex [\MyBlue] (l1) at (0,0) {}; 
    \vertex [\MyBlue] (l2) at (1,0) {};
    \vertex [\MyBlue] (l3) at (2,0) {};
    \vertex [\MyBlue] (l4) at (3,0) {};
    \vertex [\MyBlue] (l5) at (4,0) {};
    
    \draw
    (l1) edge [\MyBlue] (l5)
    (l1) edge (a)
    (l2) edge (b)
    (l3) edge (c)
    (l4) edge (d)
    (l5) edge (e)
    (-2,0) edge node [label=below:$e_1$] {} (l1)
    (-2,0) edge [\MyBlue] (l1)
    (l5) edge node [label=below:$e_5$] {} (6,0)
    (l5) edge [\MyBlue](6,0)
    ;
    
    \filldraw[fill=black] (-2,0) -- (-3,1) -- (-3,-1) -- (-2,0);
    \filldraw[fill=black] (6,0) -- (7,1) -- (7,-1) -- (6,0);
    \end{scope}
    
    \begin{scope} [shift={(8,-2)}]
        \draw (0,0) edge [->] (2,0);
    \end{scope}
    
    \begin{scope}[shift={(15,0)}]
    \node (F1) at (-4,0) {$T_1$};  
    \node (F2) at (2,1) {$C$};  
    \vertex [\MyRed, label=below:$x_1$] (a) at (0,-1) {};
    \vertex [\MyRed, label=below:$x_2$] (b) at (1,-1) {};
    \vertex [\MyRed, label=below:$x_3$] (c) at (2,-1) {};
    \vertex [\MyRed, label=below:$x_4$] (d) at (3,-1) {};
    \vertex [\MyRed, label=below:$x_5$] (e) at (4,-1) {};
    
    \vertex [\MyRed] (l1) at (0,0) {}; 
    \vertex [\MyRed] (l2) at (1,0) {};
    \vertex [\MyRed] (l3) at (2,0) {};
    \vertex [\MyRed] (l4) at (3,0) {};
    \vertex [\MyRed] (l5) at (4,0) {};
    
    \draw
    (l1) edge [\MyRed] (l5)
    (l1) edge [\MyRed] (a)
    (l2) edge [\MyRed] (b)
    (l3) edge [\MyRed] (c)
    (l4) edge [\MyRed] (d)
    (l5) edge [\MyRed] (e)
    (-2,0) edge node [label=below:$e_1$] {} (l1)
    (-2,0) edge (l1)
    (l5) edge node [label=below:$e_5$] {} (6,0)
    (l5) edge (6,0)
    ;
    
    \filldraw[fill=black] (-2,0) -- (-3,1) -- (-3,-1) -- (-2,0);
    \filldraw[fill=black] (6,0) -- (7,1) -- (7,-1) -- (6,0);
    \end{scope}

    \begin{scope}[shift={(15,-4)}]
    \node (F1) at (-4,0) {$T_2$};  
    \node (F2) at (2,1) {$C$};  
    \vertex [\MyRed, label=below:$x_1$] (a) at (0,-1) {};
    \vertex [\MyRed, label=below:$x_2$] (b) at (1,-1) {};
    \vertex [\MyRed, label=below:$x_3$] (c) at (2,-1) {};
    \vertex [\MyRed, label=below:$x_4$] (d) at (3,-1) {};
    \vertex [\MyRed, label=below:$x_5$] (e) at (4,-1) {};
    
    \vertex [\MyRed] (l1) at (0,0) {}; 
    \vertex [\MyRed] (l2) at (1,0) {};
    \vertex [\MyRed] (l3) at (2,0) {};
    \vertex [\MyRed] (l4) at (3,0) {};
    \vertex [\MyRed] (l5) at (4,0) {};
    
    \draw
    (l1) edge [\MyRed] (l5)
    (l1) edge [\MyRed] (a)
    (l2) edge [\MyRed] (b)
    (l3) edge [\MyRed] (c)
    (l4) edge [\MyRed] (d)
    (l5) edge [\MyRed] (e)
    (-2,0) edge node [label=below:$e_1$] {} (l1)
    (-2,0) edge (l1)
    (l5) edge node [label=below:$e_5$] {} (6,0)
    (l5) edge (6,0)
    ;
    
    \filldraw[fill=black] (-2,0) -- (-3,1) -- (-3,-1) -- (-2,0);
    \filldraw[fill=black] (6,0) -- (7,1) -- (7,-1) -- (6,0);
    \end{scope}
    \end{tikzpicture}
    \caption{\stevenrevise{$B$ (green) and $B'$ (blue) are examples of bypass blocks. As soon as there is at least
    one bypass block in $F_C$, all the taxa in $C$ are necessarily singletons (shown in orange). After splitting all bypass blocks these singletons can be merged into a single block $C$, show in in red.}}
    \label{fig:bypass_split}
    \end{center}
\end{figure}
\vspace{0.4cm}
    
    The total number of times a block is split is at most $t$ and after that we merged $t+1$ singleton blocks into a single block. Hence, the size of $\Fu_C'$ is at most $\ku$.

\end{proof}

\subsection{Tightness of the Truncation Length} %
\label{subsec:tightness}

It is natural to ask whether the truncation length $ \rvalue$ used in the chain reduction can be improved i.e. reduced in size. We show that for various combinations of $t$ and $k$ the bound is tight. (For $t=2$ reducing chains to length 3 is already well-known to be tight).

Our first family of tight examples is as follows; we show for every $t \geq 3$ that there is an instance where reducing the $t+1$ term to $t$ in $\rvalue$ is not safe.

Let $t \geq 3$. Construct $t$ different trees $T_1, \ldots T_t$ by starting with a $(t+1)$-chain and denote the taxa in this chain $1, 2, \ldots, t, t+1$. Then for each $T_i$ in $\T$ insert a common chain $C' = (x_1, x_2, \cdots, x_m)$ after the $i$th taxon in the $t$-chain; we let $m=2(t+2)$. Finally, in \stevenrevise{$T_2$} (but not in any of the other trees) we reverse the order of the chain $C'$, resulting in $x_m$ being sibling to taxon 1 in that tree. See Figure \ref{FIG_unrooted_revised}.

Note that if we truncate $C'$ to a chain $C$ of length $t$, then the resulting modified set of trees has an agreement forest of size at most $t+1$: we take the taxa from the $(t+1)$\stevenrevise{-}chain as a single block, together with $|C|=t$ singleton blocks. We will show that the original trees did not have an agreement forest of size $t+1$ or less. From this we will conclude that taking a truncation length of $\min \{ \max \{ k,3\}, t \}$ and parameter $k=(t+1)$ is not safe, because the first set of trees do not have an agreement forest of size at most $t+1$, but the reduced trees do.

Consider then the original set of trees. Note firstly that if any block $B$ of an agreement forest of the original trees contains two or more taxa from $C'$, then no two taxa from $\{1, 2, \ldots, t, t+1\}$ can be in the same block of the forest. Moreover, if $|B \cap C'| \geq 3$, then $B$ cannot contain any taxa from  $\{1, 2, \ldots, t, t+1\}$, meaning that the forest has at least $t+2$ blocks. So suppose there is no block $B$ with the property that $|B \cap C'| \geq 3$. Then due to the choice of $m = 2(t+2)$ it follows that the forest contains at least $t+2$ blocks; we are done.

\begin{figure}[h]
    \begin{center}
    \begin{tikzpicture}[scale = 0.4]

    \node (F1) at (-4.5,0) {$T_1$};    
    \vertex [label=left:$1$] (1) at (-2,0) {};
    \vertex [label=below:$2$] (2) at (1,-1) {};
    \vertex [label=below:$3$] (3) at (2,-1) {};
    \vertex [label=below:$4$] (4) at (3,-1) {};
    \vertex [label=below:$t$] (t') at (5,-1) {};
    \vertex [label=right:$t+1$] (t) at (6,0) {};
    
    \node (C1) [draw, line width = 1pt, minimum width=0.5cm, minimum height=0.5cm] at (-0.5, 0) {$\overrightarrow{C'}$};
    
    \vertex (l1) at (1,0) {}; 
    \vertex (l2) at (2,0) {};
    \vertex (l3) at (3,0) {};
    \vertex (l7) at (5,0) {};
    
    \draw
    (1) edge (C1)
    (C1) edge (l3)
    (l3) edge[dotted] (l7)
    (l1) edge (2)
    (l2) edge (3)
    (l3) edge (4)
    (l7) edge (t')
    (l7) edge (t)
    ;
    
    \node (F2) at (-4.5,-3.5) {$T_2$};    
    \vertex [label=left:$1$] (1) at (-2,-3.5) {};
    \vertex [label=below:$2$] (2) at (-1,-4.5) {};
    \vertex [label=below:$3$] (3) at (2,-4.5) {};
    \vertex [label=below:$4$] (4) at (3,-4.5) {};
    \vertex [label=below:$t$] (t') at (5,-4.5) {};
    \vertex [label=right:$t+1$] (t) at (6,-3.5) {};
    
    \node (C1) [draw, line width = 1pt, minimum width=0.5cm, minimum height=0.5cm] at (0.5, -3.5) {$\overleftarrow{C'}$};
    
    \vertex (l1) at (-1,-3.5) {}; 
    \vertex (l2) at (2,-3.5) {};
    \vertex (l3) at (3,-3.5) {};
    \vertex (l7) at (5,-3.5) {};
    
    \draw
    (1) edge (C1)
    (C1) edge (l3)
    (l3) edge[dotted] (l7)
    (l1) edge (2)
    (l2) edge (3)
    (l3) edge (4)
    (l7) edge (t')
    (l7) edge (t)
    ;
    
    \node (F3) at (-4.5,-7) {$T_3$};    
    \vertex [label=left:$1$] (1) at (-2,-7) {};
    \vertex [label=below:$2$] (2) at (-1,-8) {};
    \vertex [label=below:$3$] (3) at (0,-8) {};
    \vertex [label=below:$4$] (4) at (3,-8) {};
    \vertex [label=below:$t$] (t') at (5,-8) {};
    \vertex [label=right:$t+1$] (t) at (6,-7) {};
    
    \node (C1) [draw, line width = 1pt, minimum width=0.5cm, minimum height=0.5cm] at (1.5, -7) {$\overrightarrow{C'}$};
    
    \vertex (l1) at (-1,-7) {}; 
    \vertex (l2) at (0,-7) {};
    \vertex (l3) at (3,-7) {};
    \vertex (l7) at (5,-7) {};
    
    \draw
    (1) edge (C1)
    (C1) edge (l3)
    (l3) edge[dotted] (l7)
    (l1) edge (2)
    (l2) edge (3)
    (l3) edge (4)
    (l7) edge (t')
    (l7) edge (t)
    ;
    
    \node (F4) at (-4.5,-11.5) {$T_t$};    
    \vertex [label=left:$1$] (1) at (-2,-11.5) {};
    \vertex [label=below:$2$] (2) at (-1,-12.5) {};
    \vertex [label=below:$3$] (3) at (0,-12.5) {};
    \vertex [label=below:$4$] (4) at (1,-12.5) {};
    \vertex [label=below:$t$] (t') at (3,-12.5) {};
    \vertex [label=right:$t+1$] (t) at (6,-11.5) {};
    
    \node (C1) [draw, line width = 1pt, minimum width=0.5cm, minimum height=0.5cm] at (4.5, -11.5) {$\overrightarrow{C'}$};
    
    \vertex (l1) at (-1,-11.5) {}; 
    \vertex (l2) at (0,-11.5) {};
    \vertex (l3) at (1,-11.5) {};
    \vertex (l7) at (3,-11.5) {};
    
    \draw
    (1) edge (l3)
    (C1) edge (t)
    (l3) edge[dotted] (l7)
    (l1) edge (2)
    (l2) edge (3)
    (l3) edge (4)
    (l7) edge (t')
    (l7) edge (C1)
    ;
    
    \draw (F1) edge[dotted] (F2);
    \draw (F2) edge[dotted] (F3);
    \draw (F3) edge[dotted] (F4);
    
    \draw (9,-6) edge[->] node [label=above:\small{Reduce $C'$}] {} (13,-6);
    \draw (9,-6) edge[->] node [label=below:\small{to $C$}] {} (13,-6);

    \node (F1) at (27,0) {$T_{C1}$};    
    \vertex [label=left:$1$] (1) at (15,0) {};
    \vertex [label=below:$2$] (2) at (18,-1) {};
    \vertex [label=below:$3$] (3) at (19,-1) {};
    \vertex [label=below:$4$] (4) at (20,-1) {};
    \vertex [label=below:$t$] (t') at (22,-1) {};
    \vertex [label=right:$t+1$] (t) at (23,0) {};
    
    \node (C1) [draw, line width = 1pt, minimum width=0.5cm, minimum height=0.5cm] at (16.5, 0) {$\overrightarrow{C}$};
    
    \vertex (l1) at (18,0) {}; 
    \vertex (l2) at (19,0) {};
    \vertex (l3) at (20,0) {};
    \vertex (l7) at (22,0) {};
    
    \draw
    (1) edge (C1)
    (C1) edge (l3)
    (l3) edge[dotted] (l7)
    (l1) edge (2)
    (l2) edge (3)
    (l3) edge (4)
    (l7) edge (t')
    (l7) edge (t)
    ;
    
    \node (F2) at (27,-3.5) {$T_{C2}$};    
    \vertex [label=left:$1$] (1) at (15,-3.5) {};
    \vertex [label=below:$2$] (2) at (16,-4.5) {};
    \vertex [label=below:$3$] (3) at (19,-4.5) {};
    \vertex [label=below:$4$] (4) at (20,-4.5) {};
    \vertex [label=below:$t$] (t') at (22,-4.5) {};
    \vertex [label=right:$t+1$] (t) at (23,-3.5) {};
    
    \node (C1) [draw, line width = 1pt, minimum width=0.5cm, minimum height=0.5cm] at (17.5, -3.5) {$\overleftarrow{C}$};
    
    \vertex (l1) at (16,-3.5) {}; 
    \vertex (l2) at (19,-3.5) {};
    \vertex (l3) at (20,-3.5) {};
    \vertex (l7) at (22,-3.5) {};
    
    \draw
    (1) edge (C1)
    (C1) edge (l3)
    (l3) edge[dotted] (l7)
    (l1) edge (2)
    (l2) edge (3)
    (l3) edge (4)
    (l7) edge (t')
    (l7) edge (t)
    ;
    
    \node (F3) at (27,-7) {$T_{C3}$};    
    \vertex [label=left:$1$] (1) at (15,-7) {};
    \vertex [label=below:$2$] (2) at (16,-8) {};
    \vertex [label=below:$3$] (3) at (17,-8) {};
    \vertex [label=below:$4$] (4) at (20,-8) {};
    \vertex [label=below:$t$] (t') at (22,-8) {};
    \vertex [label=right:$t+1$] (t) at (23,-7) {};
    
    \node (C1) [draw, line width = 1pt, minimum width=0.5cm, minimum height=0.5cm] at (18.5, -7) {$\overrightarrow{C}$};
    
    \vertex (l1) at (16,-7) {}; 
    \vertex (l2) at (17,-7) {};
    \vertex (l3) at (20,-7) {};
    \vertex (l7) at (22,-7) {};
    
    \draw
    (1) edge (C1)
    (C1) edge (l3)
    (l3) edge[dotted] (l7)
    (l1) edge (2)
    (l2) edge (3)
    (l3) edge (4)
    (l7) edge (t')
    (l7) edge (t)
    ;
    
    \node (F4) at (27,-11.5) {$T_{Ct}$};    
    \vertex [label=left:$1$] (1) at (15,-11.5) {};
    \vertex [label=below:$2$] (2) at (16,-12.5) {};
    \vertex [label=below:$3$] (3) at (17,-12.5) {};
    \vertex [label=below:$4$] (4) at (18,-12.5) {};
    \vertex [label=below:$t$] (t') at (20,-12.5) {};
    \vertex [label=right:$t+1$] (t) at (23,-11.5) {};
    
    \node (C1) [draw, line width = 1pt, minimum width=0.5cm, minimum height=0.5cm] at (21.5, -11.5) {$\overrightarrow{C}$};
    
    \vertex (l1) at (16,-11.5) {}; 
    \vertex (l2) at (17,-11.5) {};
    \vertex (l3) at (18,-11.5) {};
    \vertex (l7) at (20,-11.5) {};
    
    \draw
    (1) edge (l3)
    (C1) edge (t)
    (l3) edge[dotted] (l7)
    (l1) edge (2)
    (l2) edge (3)
    (l3) edge (4)
    (l7) edge (t')
    (l7) edge (C1)
    ;
    
    \draw (F1) edge[dotted] (F2);
    \draw (F2) edge[dotted] (F3);
    \draw (F3) edge[dotted] (F4);
        
    \end{tikzpicture}
    \caption{If we take the common chain $C'$ to have length \stevenrevise{2(t+2)} then truncating it to a length $t$ common chain $C$ alters the size of a maximum agreement forest from at least $t+2$ to at most $t+1$.}
    \label{FIG_unrooted_revised}
    \end{center}
\end{figure}

\newpage

Our second family of tight examples shows that for every $k\geq 4$ a truncation length of
$\min \{ \max \{ k-1,3\}, t+1 \}$ is also not safe. The construction is similar to the previous case,
except that we take $m=2k+1$, $t = k+2$ and the chain $C'$ in \stevenrevise{$T_2$} is not reversed. \stevenlater{That is, there are $t=k+2$ trees and each of these trees is obtained by inserting the chain $C'$ into a starting chain comprising $t+1 = k+3$ taxa}. After truncation of the common chain to length $\min \{ \max \{ k-1,3\}, t+1 \} = k-1$, observe that the reduced trees have an agreement forest of size $k$ (i.e. the original taxa $\{1, 2, \ldots, t, t+1\}$ in one block, and $k-1$ singleton blocks corresponding to taxa in $C$). Before truncation, note that if a block $B$ has the property that $|B \cap C'| \geq 3$, then the taxa $\{2, \ldots, t\}$ must be in separate blocks, meaning that there will be at least $t-1 = k+1 > k$ blocks in the forest. But if no such block $B$ exists, then there will be at least $(2k + 1)/2 > k$ blocks. Hence, prior to truncation, an agreement forest contains at least %
\leo{$k+1$}
blocks.

\section{Bounding the \stevenfinal{Number of Leaves in Each Tree}}\label{sec:kern}

    \begin{Lemma}\label{LEM taxa bound}
    Let~$\T$ be a set of~$t$ unrooted binary phylogenetic trees on~$X$ and~$\Fu$ an unrooted agreement forest for~$\T$ of size~$\leonewer{k\geq 2}$. If neither the subtree nor the chain reduction are applicable in~$\T$, then for each block~$B\in\Fu$, it holds that
    \[
        |B| \leq 2r\sum_{T\in\T} deg^T(B) - 3r
    \]
    with $\rvalue$ and~$t\geq 2$.
\end{Lemma}

\begin{proof}
    Let~$B\in\Fu$ and $d=\sum_{T\in\T} deg^T(B)$. Note that if $|B|=1$ then $d = t \geq 2$ and the right hand side of the claimed inequality evaluates to at least $4r - 3r = r \geq 3$, so the claim holds immediately. Hence we can focus on $|B| \geq 2$.
    
    For an arbitrary tree $T^*\in\T$, consider the tree~$T_B$ obtained from~$T^*|B$ as follows. For an  edge~$e$ of~$T^*|B$ and tree~$T\in\T$, let $n(e,T)$ be the number of internal vertices on the path in~$T[B]$ corresponding to\footnote{\stevenrevise{If $e=\{u,v\}$ then the path corresponding to $e$ in $T[B]$ is the path that starts at $u$ and ends at $v$. The interior vertices of this path are the degree-2 vertices that are suppressed in order to create $e$ in $T|B$.}} edge~$e$. Subdivide each edge~$e$ of~$T^*|B$ by $\sum_{T\in\T}n(e,T)$ degree-2 vertices. Call the obtained tree~$T_B^{\mbox{full}}$. Then, in~$T_B^{\mbox{full}}$, delete all leaves and suppress their parents if they become degree-$2$. \stevenlater{(Note however that we do not suppress vertices that were degree 2 \emph{before} the deletion of the leaves)}. This gives tree~$T_B$, see Figure~\ref{fig:component_bound} for an example.

    \begin{figure}
        \centering
        \includegraphics[width=\textwidth]{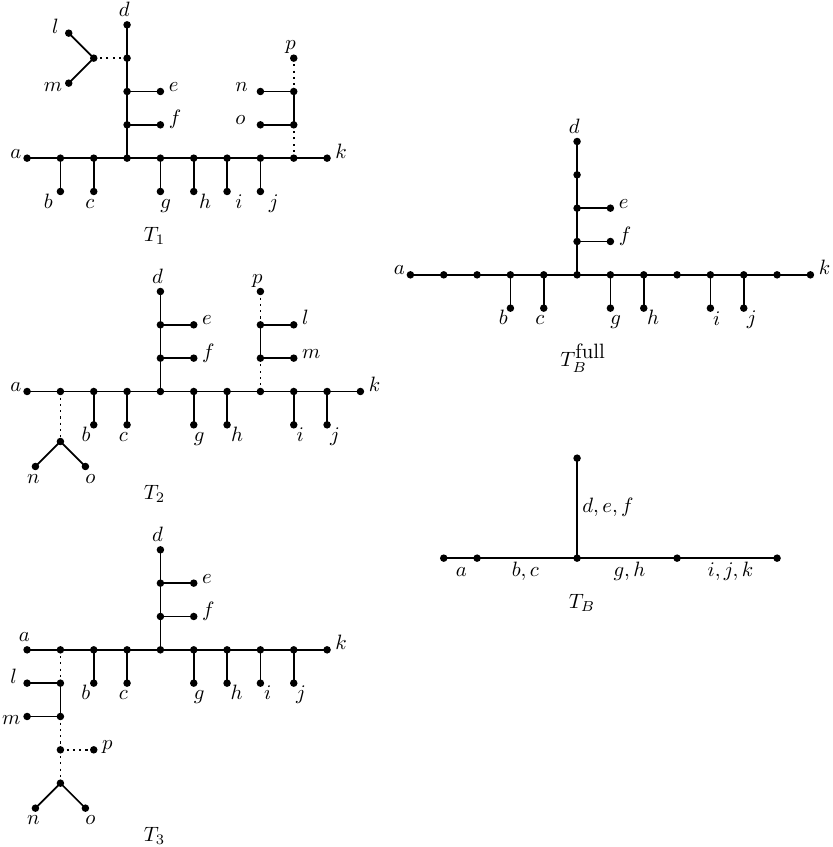}
        \caption{Example of the construction of trees~$T_B^{\mbox{full}}$ and~$T_B$ in the proof of Lemma~\ref{LEM taxa bound} with~$B=\{a,b,\ldots ,k\}$. Each edge of~$T_B$ is labelled by the leaves that are charged to it. The edges labelled~$b,c$ and~$g,h$ are considered in the first case of the proof, the edges labelled~$d,e,f$ and~$i,j,k$ are considered in the second case, while the edge labelled~$a$ is considered in the third case.}
        \label{fig:component_bound}
    \end{figure}

    Observe that the total number of degree-$2$ vertices and leaves of~$T_B$ is~$d$ and that~$\leonewer{d\geq t\geq 2}$. %

    We prove by induction on~$d$ that a tree~$S$ with~$d$ vertices of degree at most~$2$ has at most~$2d-3$ edges. If~$d=2$ then~$S$ has exactly one edge and we are done. If~$d\geq 3$, let~$S'$ be the result of deleting a leaf and suppressing its neighbour if it becomes degree-$2$. This reduces the %
    \leonewer{number of vertices of degree at most~$2$}
    by~$1$. %
    By induction, $S'$ has at most $2(d-1)-3=2d-5$ edges and therefore~$S$ has at most~$2d-3$ edges. Hence, $T_B$ has at most~$2d-3$ edges.

    We now show that the number of leaves of~$T_B^{\mbox{full}}$ is at most~$rm$ with~$m$ the number of edges of~$T_B$. Consider any maximal %
    chain~$C$ of~$T_B^{\mbox{full}}$. Then the leaves of~$C$ are deleted from~$T_B^{\mbox{full}}$ in the construction of~$T_B$ and their neighbours are suppressed if they get degree-$2$. Observe that it is not possible that a neighbour of a leaf had degree-$3$ in~$T_B^{\mbox{full}}$ and becomes degree-$1$ in~\steven{$T_B$}
    because then the subtree reduction would have been applicable in~$\T$. Hence, we can distinguish the following three cases. \leorevise{In each case we will map or ``charge'' each leaf to an edge of~$T_B$. We will do this in such a way that leaves from different maximal chains of~$T_B^{\mbox{full}}$ are charged to different edges of~$T_B$, which will make it possible to bound the total number of leaves.}
    
    The \textbf{first case} is that all neighbours of leaves in~$C$ are suppressed. Let~$e$ be the edge of~$T_B$ created by suppressing these neighbours. Charge all leaves of~$C$ to edge~$e$.
    
    \steven{The \textbf{second case} is that~$|C|\geq 2$
    \leonew{and there is at least one leaf} $x$ in $C$ such that the neighbour $v$ of $x$ in~$T_B^{\mbox{full}}$ is a
    degree-$2$ vertex, and thus is not suppressed when $x$ is deleted. \leonew{Note that there can be at most two such leaves since~$C$ is a chain.}
   If there are exactly \emph{two} such leaves, %
   then~$B$ is a common chain,
   \leonew{so~$B=C$,} and $d=2$.}
   \leonew{In this case,~$T_B$ has exactly one edge and we charge all leaves of~$B$ to it.}
    \steven{%
    \leonew{If} there is exactly \emph{one} such leaf~$x$, then}
    all other neighbours of leaves in~$C$ are suppressed, creating an edge~$e$. %
    Charge all the leaves of~$C$, \leonew{including~$x$}, to~$e$. 
    
    It remains to consider the \textbf{third case}: that~$C=\{x\}$ and the neighbour~$v$ of~$x$ in~$T_B^{\mbox{full}}$ has degree~$2$. Then~$v$ is not suppressed. Consider the edge~$e$ of~$T_B^{\mbox{full}}$ incident to~$v$ but not to~$x$. Charge~$x$ to~$e$. Observe that no other leaves are charged to~$e$ because otherwise~$C$ would not have been maximal.
    
    In general, leaves from different maximal chains of~$T_B^{\mbox{full}}$ are never charged to the same edge of~$T_B$. Moreover, each maximal chain contains at most~$r$ leaves because otherwise the chain reduction would have been applicable in~$\T$. Hence, it follows that the number of leaves of~$T_B^{\mbox{full}}$ is at most the number of edges of~$T_B$ times~$r$.
    
    This concludes the proof since~$|B|$ is equal to the number of leaves of~$T_B^{\mbox{full}}$, which we have shown to be at most~$r$ times the number of edges of~$T_B$, which we have shown to be at most~$2d-3$.
\end{proof}

\begin{Lemma}\label{LEM degree bound}
    For every $T\in\T$ the sum of the degrees of each block $B$ in $\Fu$ is bounded as follows, where $k$ is the size of $F$:
    \[
        \sum_B deg^T(B) \leq 2\ku-2.
    \]
\end{Lemma}

    \begin{proof}
\stevenlater{Let $T$ be any tree in $\T$ and let $B$ be a block of $\Fu$. Consider $T[B]$. If $T$ has an edge $e$ such that one endpoint $u$ of $e$ is a node of $T[B]$ and the other endpoint $v$ is not, and $v$ does not lie on $T[B']$ for any $B' \in \Fu$, then $v$ is necessarily not a leaf (because every taxon is in some block of the forest), and thus has degree larger than 1. If we expand $T[B]$ with the edge $e$, then the degree of this expanded block (i.e. the number of edges that have exactly one endpoint on the expanded block) is at least as large as the degree of $B$ in $T$: this is because there is at least one edge incident to $v$ that is not equal to $e$. We iterate this procedure to exhaustion i.e. until the point that no block $B \in \Fu$ can be further expanded in $T$. At this point, there will be exactly $k-1$ edges in $T$ that are not part of any expanded block, and each such edge will have its endpoints on two distinct expanded blocks. Hence, the total degree of the expanded blocks in $T$ will be exactly $2(k-1)$. By construction, this is an upper bound on the sum of the degrees of the original blocks in $T$, so 
$\sum_B deg^T(B) \leq 2\ku-2$.}   
    \end{proof}

\begin{Theorem} \label{THM kernel size}
    Let $\T$ be a set of~\leorevise{$t$} unrooted binary phylogenetic trees on $X$. \leonew{If there exists an unrooted agreement forest of~$\T$ of size at most~$k$ and neither the subtree nor chain reduction is applicable, then every tree in $\T$ has at most $\uks$ taxa, with $\rvalue$.}
\end{Theorem}

\begin{proof}
    \leonew{Let~$\Fu$ be an unrooted agreement forest of~$\T$ of size at most~$k$. We can bound the number of leaves in each tree as follows.}
    \begin{align*}
        |X|   &\leq \sum_{B\in\Fu} |B|\\
            &\leq \sum_B \left( 2r\sum_T deg^T(B) -3r\right) &(\text{\LEM{taxa bound}}) \\
            &= 2r\sum_T  \sum_B deg^T(B) - \sum_B 3r\\
            &= 2r\sum_T \sum_B deg^T(B) - 3r\ku\\
            &\leq 2r\sum_T (2\ku - 2) - 3r\ku &(\text{\LEM{degree bound}})\\
            &= t2r(2\ku-2) - 3r\ku \\
            &= 4tr\ku - 4tr - 3r\ku.
    \end{align*}
\end{proof}
We note that for $t=2$, at which point $r=3$, the above bound becomes $24k - 24 - 9k = 15k-24$. This matches
exactly with the analysis in \cite{KelkL18} which shows that for these two reduction rules
the kernel has at most $15d_{\text{TBR}}-9$ taxa where $d_{\text{TBR}}$, the \emph{Tree Bisection and Reconnect} distance, is defined to be equal to the size of a maximum agreeement forest \emph{minus one}. Their analysis is tight: they show two trees for which the kernel has exactly $15d_{\text{TBR}}-9$ taxa. Hence, for $t=2$, our analysis is also tight.

\subsection{\stevenfinal{Adapting to Rooted Trees}}
\label{subsec:rootedkern}

It is well-known that the subtree reduction rule holds in the case of multiple rooted trees. The chain reduction rule is also safe for the case of multiple rooted trees, as the next lemma shows.

\begin{Lemma}\label{LEM safe chain reduction rooted}
    Let $\T$ be a set of rooted \leo{binary} phylogenetic trees on $X$ %
    \leo{and~$C'$ a common chain of~$\T$}. Let $\T_C$ be the result of applying the chain reduction rule. Then there exists a rooted agreement forest of $\T_C$ of size at most $\kr$ if and only if $\T$ has a rooted agreement forest of size at most $\kr$.
\end{Lemma}

\begin{proof}
    In the proof for the unrooted variation of the chain reduction rule, \stevenrevise{Lemma~\ref{LEM safe chain reduction}}, there were three main cases to discuss: \stevenlater{$\Fu_C$} had 2, 1 or 0 inside-out blocks. A careful re-reading of that proof shows that all these cases and their various subcases go through in the rooted case.
\end{proof}

 Let $\Fr$ be a rooted agreement forest. We define the \emph{outdegree} of a block~$B$ in~$T$, $outdeg^T(B)$, to be the number of edges of $T$ which have their tail on $T[B]$ but not their head.  
We state the following modification of \LEM{degree bound} without proof.

\begin{Lemma}\label{root:newguy}
    For every $T\in\T$ the sum of the outdegrees of each block $B$ in $\Fr$ is bounded as follows, where $\kr$ is the size of $\Fr$:
    \[
        \sum_B outdeg^T(B) \leq \kr-1.
    \]
\end{Lemma}

\LEM{taxa bound} also requires some adjustment. Consider a block $B$ of $\Fr$ where $T|B$ has $c\geq 1$ cherries (here $T$ is an arbitrary tree from $\T$). \stevenfinal{To avoid triggering the subtree reduction it must hold that, for each such cherry $\{x,y\}$, at 
least one tree $T \in \T$ has the property that some edge of $T$ not within $T[B]$ has its tail on the \leonew{(undirected)} path \leonew{between} $x$ \leonew{and} $y$ in $T$.}
Hence, $\sum_{T} outdeg^T(B) \geq c$. If $\sum_{T} outdeg^T(B) = c$ then it is possible that $B$ has up to $r(2c-1)$ taxa. This is because a rooted binary tree on $c$ leaves has $2(c-1) = 2c-2$ edges, each of which %
\leonew{can correspond}
to an $r$-chain, but $B$ can also have an extra $r$-chain that feeds into the top of the tree.
This is summarized in the next lemma; note that the lemma only holds for non-singleton blocks, as $\sum_T outdeg^T(B) = 0$ for singleton blocks. See also Figure~\ref{fig:rootedlem}.

\begin{figure}
    \centering
    \includegraphics[width=\linewidth]{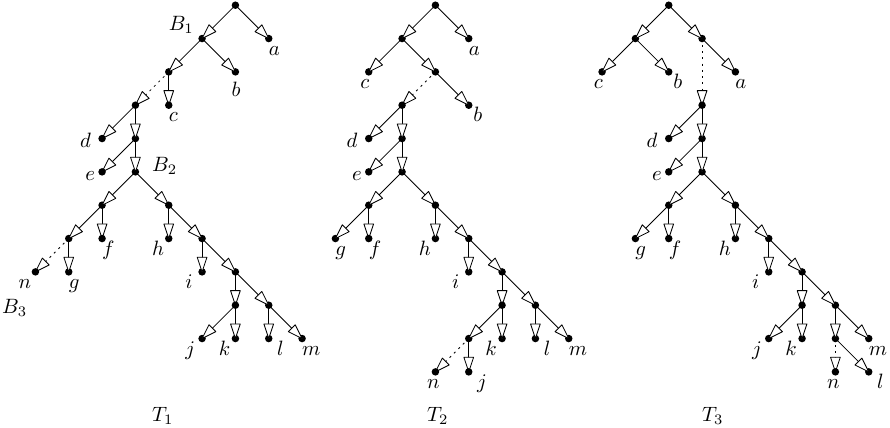}
    \caption{Illustration of Lemma~\ref{LEM rooted taxa bound}. Block~$B_2=\{d,e,\ldots ,m\}$ has total outdegree, over all trees, of~$3$, so it may contain up to~$5r$ taxa. For ease of illustration, we used chains of length~$2$, leading to~$|B_2|=10$.}
    \label{fig:rootedlem}
\end{figure}

\begin{Lemma}\label{LEM rooted taxa bound}
    The number of taxa in a non-singleton block $B$ of $\Fr$ is at most
    \begin{align*}
        |B| &\leq  2r \sum_T outdeg^T(B) - r
    \end{align*}
    with $\rrvalue$ \leorevise{and~$t$ the number of trees}.
\end{Lemma}

This means that \THM{kernel size} can also be adjusted to incorporate rooted agreement forests and provide a very similar result. 
\begin{Theorem} \label{THM rooted kernel size}
    Let $\T$ be a set of~\leorevise{$t$} rooted \leo{binary} phylogenetic trees on $X$. 
    \leonew{If there exists a rooted agreement forest for~$\T$ of size at most~$k_r$ and neither the subtree nor chain reduction is applicable, then}
    every tree in $\T$ has at most \stevenlater{$\rks$} taxa, \leonew{with $\rrvalue$.}
\end{Theorem}

\begin{proof}
\leonew{Consider a rooted agreement forest~$\Fr$ of size at most~$k_r$. We may assume that at least one block of~$\Fr$ has size at least~$2$ since otherwise we can simply merge two arbitrary blocks. We can then bound the number of leaves in each tree as follows.}

    \begin{align*}
        |X|   &\leq \sum_{B\in\Fr} |B|\\
&= \sum_{|B|=1} 1 + \sum_{|B|>1} |B|\\
            &\leq \kr-1 + \sum_{|B|>1} \left(  2r \sum_T outdeg^T(B) - r \right) & (\text{\LEM{rooted taxa bound}}) \\
&\leq \kr-1 + \sum_{B} \left(  2r \sum_T outdeg^T(B) \right) - r & \\
            &\leq \kr-1 + 2r\sum_T (\kr - 1) -r & (\text{Lemma \ref{root:newguy}})\\
            &= \kr-1 + 2tr(\kr-1) - r\\
            &= (2tr+1)(\kr-1) -r.
    \end{align*}
\end{proof}

\stevenlater{For $t=2$ and thus $r=3$ this gives $13\kr - 16$. We note that when restricted to the subtree and chain reduction rule the analysis in  \cite{kelk2023cyclic} also yields a bound of $13\kr-16$ for $t=2$; more precisely, $13d_{\text{rSPR}} - 3$ where 
$d_{\text{rSPR}}$ (the \emph{rooted Subtree Prune and Regraft} distance) is equal to $\kr - 1$. Their bound is tight\footnote{\stevenlater{The construction shown in Figure 3 of \cite{kelk2023cyclic} to demonstrate tightness for the three reduction rules considered there, can be adapted to show tightness when only the subtree and chain reduction rule are considered. This can be achieved by placing 3 taxa on each ``edge side'' of the underlying generator.}}, including additive terms, and thus so is ours.}

\emph{Remark.} The second family of tight examples given for unrooted trees in Section~\ref{subsec:tightness}, which shows that the $k$ cannot be reduced to $k-1$ in $\rvalue$, works unchanged
for rooted trees; the trees can \stevenlater{simply} be rooted on the edge \leo{incident to} taxon $1$.

\stevenrevise{
However, a slightly more complex construction is required to show that for $t\geq 2$ rooted trees, $t+1$ cannot be reduced to $t$ in $\rvalue$.   We construct the following set of $t$ trees. Every tree contains a chain $C'$ with $m = 2t+1$ taxa, and $2t$ other distinct taxa labeled $x_i$ and $y_i$, $1  \leq i  \leq t$. To construct $T_i$ we start by taking a cherry on $\{x_1, y_1\}$, placing this on the left side of the root, and on the right side of the root we place the remaining taxa
$x_2, y_2, x_3, y_3, \ldots, x_t, y_t$ in a chain (so there are two cherries in total). We then cyclically increment the indices of the labels $(i-1)$ times. Note in particular that applying a cyclical increment to $x_{t}$ and $y_{t}$ sends them back to $x_1$ and $y_1$ respectively. We complete the construction of $T_i$ by inserting the chain $C'$ between taxa $x_{i+1}$ and $y_{i+1}$ (for $i<t$) or
between taxa $x_{1}$ and $y_{1}$ (for $i=t$). The construction is summarized in
Figure \ref{FIG_rooted_tight}. Note that for clarity the figure assumes $t\geq 3$, but the construction also works for $t=2$.
}

\stevenrevise{
For $t=2$, truncating $C'$ to length $2$ reduces the size of a maximum agreement forest from 4 to 3 (we omit details), so truncation to length $t$ is not safe for $t=2$ (at $k=3$). So, we assume $t \geq 3$. It is easy to verify that, if the chain $C'$ is truncated to length $t$, the resulting trees have an agreement forest with at most $2t$ blocks: the $t$ taxa from the truncated chain become singletons, together with blocks $\{x_i, y_i\}$ for $1 \leq i \leq t$. We will show that prior to truncation an agreement forest has at least $2t+1$ blocks. From this we will conclude that a truncation length of $\min \{ \max \{ k,3\}, t \}$  is not safe, because when taking $k=2t$ (noting that $\min \{ \max \{ 2t,3\}, t \} = t$)  the original trees do not have an agreement forest of size at most $2t$, but the reduced trees do.
So, consider a maximum agreement forest of the
original trees. If there is no block $B$ such that
$|B \cap C'| \geq 2$, then there are at least $m = 2t+1$ blocks in the forest, so we are done. Hence, assume that
there is at least one block $B$ where $|B \cap C'| \geq 2$ and let $c, d$ be distinct taxa in $B \cap C'$. Observe that $B$ cannot contain any taxon from outside $C'$. To see why, suppose it contained $x_j$ for some $j$ (the same reasoning will hold for $y_j$). In some tree $x_j$ appears in the cherry left of the root, i.e. `above' $C'$, while in some tree $x_j$ appears in the cherry `underneath' $C'$: these two trees induce different topology subtrees on $\{c,d,x_j\}$, violating the definition of an agreement forest. Moving on, it follows that no two taxa from outside $C'$ can be in the same block $B' \neq B$ of the agreement forest. This is because, by construction, for every two distinct taxa in $\{x_1, y_1, ..., x_t, y_t\}$ there exists some tree $T_j$ in which one of the two taxa is above $C'$, and the other taxon is below $C'$, so $T_j[B]$ and $T_j[B']$ would intersect, again contradicting the definition of an agreement forest. This means that the taxa $\{x_1, y_1, ..., x_t, y_t\}$ induce $2t$ singleton blocks in the forest. Counting $B$ too, that yields at least $2t+1$ blocks.
}

\begin{figure}[h]
    \begin{center}
    \begin{tikzpicture}[scale = 0.4]

    \node (F1) at (-3,0) {$T_1$};
    \vertex (0) at (0,0) {};

    \vertex [label=below:$x_1$] (3) at (-3,-3) {};
    \vertex [label=below:$y_1$] (4) at (-1,-3) {};
    \vertex [label=below:$x_2$] (1) at (1,-3) {};
    \node (C1) [draw, line width = 1pt, minimum width=0.5cm, minimum height=0.5cm] at (3.5, -3.5) {$\overrightarrow{C'}$};
    \vertex [label=below:$y_2$] (2) at (4,-6) {};

    \vertex [label=below:$x_t$] (x) at (6,-8) {};
    \vertex [label=below:$y_t$] (y) at (8,-8) {};

    \vertex (c) at (-2,-2) {};
    \vertex (r1) at (2,-2) {};
    \vertex (r4) at (5,-5) {};
    \vertex (re) at (7,-7) {};
    
    \draw
    (0) edge[->] (c) edge[->] (r1)
    (c) edge[->] (3) edge [->] (4)
    (r1) edge[->] (C1) edge[->] (1)
    (C1) edge[->] (r4)
    (r4) edge[->] (2) edge[->, dotted] (re)
    (re) edge[->] (x) edge[->] (y)
    ;
    
    \node (F1) at (9,0) {$T_2$};
    \vertex (0) at (12,0) {};

    \vertex [label=below:$x_2$] (3) at (9,-3) {};
    \vertex [label=below:$y_2$] (4) at (11,-3) {};
    
    \vertex [label=below:$x_3$] (1) at (13,-3) {};
    \node (C1) [draw, line width = 1pt, minimum width=0.5cm, minimum height=0.5cm] at (15.5, -3.5) {$\overrightarrow{C'}$};
    \vertex [label=below:$y_3$] (2) at (16,-6) {};

    \vertex [label=below:$x_{1}$] (x) at (18,-8) {};
    \vertex [label=below:$y_{1}$] (y) at (20,-8) {};

    \vertex (c) at (10,-2) {};
    \vertex (r1) at (14,-2) {};
    \vertex (r4) at (17,-5) {};
    \vertex (re) at (19,-7) {};
    
    \draw
    (0) edge[->] (c) edge[->] (r1)
    (c) edge[->] (3) edge [->] (4)
    (r1) edge[->] (C1) edge[->] (1)
    (C1) edge[->] (r4)
    (r4) edge[->] (2) edge[->, dotted] (re)
    (re) edge[->] (x) edge[->] (y)
    ;

    \node (F1) at (21,0) {$T_t$};
    \vertex (0) at (24,0) {};

    \vertex [label=below:$x_{t}$] (3) at (21,-3) {};
    \vertex [label=below:$y_{t}$] (4) at (23,-3) {};
    \vertex [label=below:$x_1$] (1) at (25,-3) {};
    \node (C1) [draw, line width = 1pt, minimum width=0.5cm, minimum height=0.5cm] at (27.5, -3.5) {$\overrightarrow{C'}$};
    \vertex [label=below:$y_1$] (2) at (28,-6) {};

    \vertex [label=below:$x_{t-1}$] (x) at (30,-8) {};
    \vertex [label=below:$y_{t-1}$] (y) at (32,-8) {};

    \vertex (c) at (22,-2) {};
    \vertex (r1) at (26,-2) {};
    \vertex (r4) at (29,-5) {};
    \vertex (re) at (31,-7) {};
    
    \draw
    (0) edge[->] (c) edge[->] (r1)
    (c) edge[->] (3) edge [->] (4)
    (r1) edge[->] (C1) edge[->] (1)
    (C1) edge[->] (r4)
    (r4) edge[->] (2) edge[->, dotted] (re)
    (re) edge[->] (x) edge[->] (y)
    ;
        
    \end{tikzpicture}
    \caption{If we take the common chain $C'$ to have length \stevenrevise{$2t+1$} then truncating it to a length $t$ common chain $C$ alters the size of a maximum agreement forest from at least $2t+1$ to at most $2t$.}
    \label{FIG_rooted_tight}
    \end{center}
\end{figure}

\section{Future Work}
Although the truncation length $r$ used in our chain reduction is in some sense tight, it is unclear whether the overall bound on the size of the kernel is tight \stevenlater{for $t>2$}. \leonew{For~$t=2$ our bounds perfectly match \stevenlater{existing}
bounds which are known to be tight under subtree and chain reduction.} 
The main question is whether for $t>2$ the counting in Lemmas \ref{LEM taxa bound} and  \ref{LEM degree bound} \stevenlater{(and their rooted analogues)} can be undertaken more carefully. For $t=2$  the introduction of generators made the counting much easier \cite{KelkL18}, and subsequently became the basis for new, more powerful reduction rules \cite{kelk2023cyclic,kelk2024deep}. Generators are in essence a static, graph-based representation of maximum agreement forests, but they do not work for $t>2$. Finding an alternative to generators for $t>2$ seems an important direction for future research.

Relatedly, it is natural to ask whether the design and deployment of new, additional reduction rules can produce a smaller kernel. A natural starting point would be to analyze the extra rules that have already been developed for $t=2$. For the rooted problem this is the  `3-2 chain reduction' described in \cite{kelk2023cyclic}. In contrast, for the unrooted problem there are eight extra reduction rules known \cite{kelk2024deep}. Although the generator machinery is not available to us in the $t>2$ regime some of these reduction rules might still be correct. In such a case the main challenge will be measuring the impact of the reduction rules on the size of the kernel.  

Currently, depending on whether $t \ll k$, $k \ll t$ or $t \approx k$ we obtain a kernel %
\leonewer{with}
$O( t^2k )$, $O( tk^2)$ or $O(t^3)$ (equivalently $O(k^3)$) leaves respectively.\footnote{\stevenrevise{Measured in bits the size of the kernel is slightly larger. This is because a reasonable encoding of $t$ binary phylogenetic trees each on $n$ leaves requires $N = \Theta(tn\log{n})$ bits.}}
In an applied context it is plausible that the first scenario (comparing a small number of highly discordant gene trees) or the second scenario (comparing a large number of broadly similar gene trees) will often occur. Possibly specialized reduction rules can be designed for each of these scenarios separately. 

\stevenrevise{Perhaps the most tantalizing open question is whether rMAF and uMAF admit a kernel whose size is %
polynomial in \emph{only} $k$. The results in the present article show that they admit a kernel whose size is at most polynomial in $k$ \emph{and} $t$. 
On the other hand, the 
branching algorithms for rMAF and uMAF with running times $O^{*}(c^k)$ in   
\cite{shi2018parameterized,shi2014algorithms} automatically imply the existence of a `trivial' kernel whose size is (single) exponential in only $k$ (see e.g. Lemma 2.2 of \cite{DBLP:books/sp/CyganFKLMPPS15} for a proof of this implication; it is a foundational fact in parameterized complexity). The number of distinct binary phylogenetic trees on $n$ leaves is bounded by a function of $n$, but the bound is super-exponential (see Proposition 2.1.4 and Corollary 2.2.4 of \cite{SempleS03}), so to obtain a polynomial kernel in $k$ it will be necessary to produce a reduced instance on at most poly($k$) trees and poly$(k)$ leaves.
A potentially relevant fact here is that, for rMAF and uMAF, there is no obvious function $f$ of $k$ whereby we can say ``If $t \geq f(k)$ then a MAF with at most $k$ blocks does not exist.'' To see this, observe that two distinct trees, each on $n/2$ leaves, can be connected together with a single edge to obtain $O(n^2)$ trees on $n$ leaves, which have a MAF of size only 2.}
\stevenrevise{Indeed, it might be that a kernel whose size is bounded by a polynomial of $k$ simply does not exist, under reasonable complexity-theoretic assumptions. The literature on kernel-size lower bounds (using techniques such as \emph{cross-composition} \cite{bodlaender2014kernelization}; see also the chapter on lower bounds in \cite{fomin2019kernelization}) could be very useful in this regard.}

\section{Acknowledgements}

Ruben Meuwese was supported by the Dutch Research Council (NWO), project OCENW.GROOT.2019.015. \stevenrevise{We thank the anonymous reviewers for their
helpful feedback. We also thank Enrico Iurlano and Johannes Varg for pointing out an error in the tightness construction in an earlier version of the article}.

\bibliography{kernelV3}{}
\bibliographystyle{plainnat}

\end{document}